\documentclass[12pt,final]{amsart}
\usepackage{amsmath}
\usepackage{amsfonts}
\usepackage{amssymb}  
\usepackage{graphicx} 

\usepackage[color,notref,notcite]{showkeys}  
\definecolor{labelkey}{rgb}{0.6,0,1}

\usepackage[normalem]{ulem}
\newcounter{corr}
\definecolor{violet}{rgb}{0.580,0.,0.827}
\newcommand{\corr}[3]{\typeout{Warning : a correction remains in page
\thepage}
				\stepcounter{corr}        
				{\color{blue}\ifmmode\text{\,\sout{\ensuremath{#1}}\,}\else\sout{#1}\fi}
       {\color{red}#2}
       {\color{violet} #3}}
\numberwithin{equation}{section}
 \newtheorem{thm}{Theorem}[section]
 \newtheorem{lem}[thm]{Lemma}%[section]
 \newtheorem{exam}[thm]{Example}%[section]
 \newtheorem{prop}[thm]{Proposition}%[section]
 \newtheorem{cor}[thm]{Corollary}%[section]
 \newtheorem{rem}[thm]{Remark}%[section]
 \newtheorem{defn}[thm]{Definition}%[section]  

 \def\Id{\mathop{\rm Id}\nolimits}
 
 \def\dist{\mathop{\rm dist}\nolimits}

 \def\span{\mathop{\rm Span}\nolimits}

\newcommand{\A}{{\mathcal A}}   % opérateur
\newcommand{\B}{{\mathcal B}}   % opérateur
\newcommand{\C}{{\mathbb C}}   % nombres complexes
 \newcommand{\R}{{\mathbb R}}           % nombres reels
\newcommand{\N}{{\mathbb N}^*}        % entiers naturels
\newcommand{\K}{{\mathbb K}}        % corps des réels ou complexes
\newcommand{\Z}{{\mathbb Z}}          % entiers relatifs
\newcommand{\U}{{\mathcal U}}   
\newcommand{\W}{{\mathcal W}}  
\newcommand{\V}{{\mathcal V}} 
\newcommand{\X}{{\mathcal X}} 
   
\newcommand{\Y}{{\mathcal Y}}

 \newcommand{\GH}{{\mathcal H}}

\author{W. Arendt}
\address{Wolfgang Arendt, Institute of Applied Analysis, University of Ulm. Helmholtzstr. 18, D-89069 Ulm (Germany)} 
\email{wolfgang.arendt@uni-ulm.de}

\author{I. Chalendar}
\address{Isabelle Chalendar,  Universit\'e Paris-Est, LAMA, (UMR 8050), UPEM, UPEC, CNRS, F-77454, Marne-la-Vallée (France)}
\email{isabelle.chalendar@u-pem.fr}

\author{R. Eymard}
\address{Robert Eymard,  Universit\'e Paris-Est, LAMA, (UMR 8050), UPEM, UPEC, CNRS, F-77454, Marne-la-Vallée (France)}
\email{robert.eymard@u-pem.fr}
\title[Galerkin approximation]{Galerkin approximation of linear problems\\ in Banach and Hilbert spaces}
\begin{document}	

\begin{abstract}   
 In this paper we study the conforming Galerkin approximation of the problem: find $u\in\U$ such that $a(u,v) = \langle L, v \rangle$ for all $v\in \V$, where $\U$ and $\V$ are  Hilbert or Banach  spaces, $a$ is  a continuous bilinear or sesquilinear form  and $L\in\V'$ a given data.  The approximate solution is sought in a finite dimensional subspace of $\U$, and test functions are taken in a finite dimensional subspace of $\V$. We provide a necessary and sufficient  condition on the form $a$ for convergence  of the Galerkin approximation, which is also  equivalent to  convergence of the Galerkin approximation for the adjoint problem. We also characterize the fact that  $\U$ has a finite dimensional Schauder decomposition in terms of properties related to the Galerkin approximation. 
 In the case of Hilbert spaces, we prove that the only bilinear or sesquilinear forms for which any Galerkin approximation converges (this property is called the \emph{universal Galerkin property}) are the \emph{essentially coercive} forms. In this case, a generalization of the Aubin-Nitsche Theorem leads to optimal a priori estimates  in terms of regularity properties of  the right-hand side $L$, as shown by several applications. 
 Finally, a section entitled "Supplement"  provides some consequences of our results for the approximation of saddle point problems.
\end{abstract}	
\maketitle
   
\section{Introduction}

Due to its practical importance, the approximation of elliptic problems in Banach or Hilbert spaces has been the object of numerous works. In Hilbert spaces, a crucial result is the simultaneous use of the Lax-Milgram theorem and of C\'ea's Lemma to conclude the convergence of conforming Galerkin methods in the case that the elliptic problem is resulting from a coercive bilinear or sesquilinear form.
 \\

But the coercivity property is lost in many practical situations: for example, consider the Laplace operator perturbed by a convection term or a reaction term (see the example in Section \ref{sec:6}), and the approximation of non-coercive forms must be studied as well. For particular bilinear or sesquilinear forms, the Fredholm alternative provides an existence result in the case where the problem is well-posed in the Hadamard sense. Such results have been extended by Banach, Ne\v cas, Babu\v ska and Brezzi in the case of bilinear forms on Banach spaces. The conforming approximation of such problems enters into the framework of the so-called Petrov--Galerkin methods, for which sufficient conditions for the convergence are classical (see for example the references \cite{AU19,CC13,EG04,XZ03} which also include the case of non-conforming approximations). 

Nevertheless, these sufficient conditions do not guarantee that for a given problem, there exists a converging Galerkin approximation. Moreover, they do not answer the following question, which is important in practice: under which conditions does the Galerkin approximation exist and converge to the solution of the continuous problem for \emph{any} sufficiently fine approximation (for example, letting the degree of an approximating polynomial or the number of modes in a Fourier approximation be high enough, or letting the size of the mesh for a finite element method be small enough, and, in the case of Hilbert spaces, using the Galerkin method and not the Petrov--Galerkin method)?

The aim of this paper is precisely to address such questions for not necessarily coercive bilinear or sesquilinear forms defined on some Banach or Hilbert spaces (we treat the real and complex cases simultaneously). We shall restrict this study to conforming approximations, in the sense that the approximation will be sought in subspaces of the underlying space, using the continuous bilinear or sesquilinear form. 
 
 In the first part we consider the Banach space framework. Given a continuous bilinear form $a:\U\times\V\to\R$ where $\U$ and $\V$ are reflexive, separable Banach spaces, one is interested in the existence and the convergence of the Galerkin approximation to $u$, where $u$ is the solution of the following problem:
 \begin{equation}\label{eq:int1}
 \hbox{Find }u\in\U\hbox{ such that }
 a(u,v)=\langle L,v\rangle, \hbox{ for all } v\in\V,    
 \end{equation}  
where $L\in\V'$ is given (the existence and uniqueness of $u$ are obtained under the Banach-Ne\v cas-Babu\v ska conditions, see for example \cite[Theorem 2.6]{EG04}). For approximating sequences $(\U_n)_{n\in\N}$, $(\V_n)_{n\in\N}$ (see Section~\ref{sec:3} for the definition), the Galerkin approximation of \eqref{eq:int1} is given by the sequence $(u_n)_{n\in\N}$ such that, for any $n\in\N$, $u_n$ is the solution of the following finite dimensional linear problem:
    \begin{equation}\label{eq:int2}
  \hbox{Find } u_n\in\U_n\hbox{ such that } a(u_n,\chi)=\langle L,\chi\rangle,  \hbox{ for all } \chi\in\V_n.
   \end{equation}   
It is known that, if $\dim\U_n = \dim\V_n$, the uniform Banach-Ne\v cas-Babu\v ska condition (BNB) given in Section~\ref{sec:3}  is sufficient for these existence and convergence properties (see for example \cite[Theorem 2.24]{EG04}). We show here that this condition is also necessary and, surprisingly, that the convergence of the Galerkin approximation of \eqref{eq:int1} is equivalent to that of the Galerkin approximation of the dual problem.

 These two results seem to be new and are presented in Section~\ref{sec:3}.

In Section~\ref{sec:new3}, we ask the following:
given a form $a$ such that (\ref{eq:int1}) is well-posed, do there always exist approximating sequences in $\U$ and $\V$ such that the Galerkin approximation converges? 
    Surprisingly, the answer is negative (even though the spaces $\U$ and $\V$ are supposed to be reflexive and separable). In fact, such approximating sequences exist if and only if  the Banach space  $\U$ has a finite dimensional Schauder decomposition, a property which is strictly more general than having a Schauder basis.

In the remainder of the paper, merely Hilbert spaces are considered and moreover we assume that $\U=\V$ and $\U_n=\V_n$ for all $n\in\N$. Given is a continuous bilinear form $a:\V\times \V\to\R$, where $\V$ is a separable Hilbert space. Assuming that (\ref{eq:int1}) is well-posed, we show that the convergence of the Galerkin approximation for \emph{all} approximating sequences in $\V$ (which we call here the \emph{universal Galerkin property}) is equivalent to $a$ being 
\emph{essentially coercive}, which means that a compact perturbation of $a$ is coercive.
This notion of essential coercivity can also be characterized by a certain \emph{weak-strong inverse continuity} of $a$, which, in fact, we take as definition of essential coercivity (Definition~\ref{def:2.1}). 

We then derive improved a priori error estimates by generalizing the Aubin--Nitsche argument to  non-symmetric forms and also allowing the given right hand side  $L$ of (\ref{eq:int1}) to belong to arbitrary interpolation spaces in between $\V$ and $\V'$. These generalizations are applied to two cases: the approximation of selfadjoint positive operators with compact resolvent (in this case, it is seen that our a priori error estimate is optimal, with the fastest speed of convergence for $L$ in $\V$, the slowest for $L\in\V'$) and the finite element approximation of a non-selfadjoint elliptic differential operator, including convection and reaction terms which is indeed essentially coercive.

	We finally give some further historical remarks in Section~\ref{sec:8}, where we consider saddle point problems. As a consequence of our results, we show that Brezzi's conditions,  implying the convergence of mixed approximations (which are the Galerkin ones in the case of saddle point problems), are also necessary for this convergence.

To avoid any ambiguity, in the sequel, we let ${\mathbb N}=\{0,1,2,\cdots\}$  and $\N={\mathbb N}\setminus\{0\}$.   \\

The paper is organized as follows:
\tableofcontents 

\section{Petrov--Galerkin approximation}\label{sec:3}
In this section we give a characterization of the convergence of Petrov--Galerkin methods, that, for short, we call Galerkin convergence. A basic definition is the following.
\begin{defn}[Approximating sequences of Banach spaces]\label{def:appseq}
	Let $\V$ be a separable Banach space.  
	An \emph{approximating sequence} of $\V$ is a sequence $(\V_n)_{n\in\N}$ of finite dimensional subspaces of $\V$ such that 
	\[  \dist (v,\V_n)\to 0\mbox{ as }n\to\infty\]
	for all $v\in\V$, where  $\dist (u,\V_n):=\inf\{ \|u-\chi\|:\chi\in \V_n\}$.    
\end{defn}

Now let $\U$ and $\V$ be two separable, reflexive Banach spaces over $\K=\R$ or $\C$ and $a: \U \times \V \to \K $  be a continuous  sesquilinear form such that 
\[   | a(u,v)|\leq M \|u\|_{\U}\| v \|_{\V}\mbox{ for all }u\in \U,v\in \V  \] 
where $M>0$ is a constant. We assume that $\U$ and $\V$ are infinite dimensional and that $(\U_n)_{n\in\N}$ and $(\V_n)_{n\in\N}$ are approximating sequences of $\U$ and $\V$ respectively.
% This means (in the first case) that  $\U_n$ is a finite dimensional subspace of $\U$ for all $n\in\N$ and that 
%\[   \lim_{n\to \infty}\dist (u,\U_n)=0\mbox{ for all }u\in\U.  \]  
 We also assume throughout that 
\[0\neq \dim \U_n=\dim \V_n    \mbox{ for all }n\in\N.\]
Given $L\in \V'$ we search a solution $u$ of the problem:  
\begin{equation}\label{eq:3.1}
\hbox{find }u\in\U\hbox{ such that }
 a(u,v)=\langle L,v\rangle, \hbox{ for all } v\in\V.
\end{equation} 
Moreover we want to approximate such a solution by $u_n$, the solution of the problem:
\begin{equation}\label{eq:3.2}
\hbox{find } u_n\in\U_n\hbox{ such that } a(u_n,\chi)=\langle L,\chi\rangle,  \hbox{ for all } \chi\in\V_n.  
\end{equation}  
Note that, given $n\in\N$,  there exists a unique $u_n\in \U_n$ satisfying (\ref{eq:3.2}) if and only if 
\begin{equation}\label{eq:3.3}
\hbox{for all }u\in\U_n, \big(a(u,\chi)=0 \mbox{ for all } \chi\in\V_n\big)\Rightarrow u=0,
\end{equation}
since, by assumption,  $\U_n$ and $\V_n$ have the same finite dimension. 

Let us briefly recall the origin of the Banach-Ne\v cas-Babu\v ska conditions for the well-posedness of (\ref{eq:3.1}) as stated for example in \cite{EG04,XZ03,AU19}  (equivalent conditions are proposed in \cite{CC13} in the case of Hilbert spaces).
Let us consider the associated operator $\A:\U\to \V'$ defined by 
\[ \langle \A u,v\rangle=a(u,v) \quad (u\in\U,v\in\V).   \]
Then $\A$ is linear, bounded with $\|\A\|\leq M$. By the Inverse Mapping Theorem, $\A$ has closed range and is injective if and only if there exists $\beta>0$ such that 
\begin{equation}\label{eq:3.4}
\|\A u\|_{\V'}\geq \beta \|u\|_\U \mbox{ for all }u\in\U.
\end{equation} 
By the definition of the norm of $\V'$, this can be reformulated by 
\begin{equation}\label{eq:3.5}
  \sup_{  \|v\|_\V\leq 1}|a(u,v)|\geq \beta \|u\|_{\U}  \mbox{ for all }u\in\U.
 \end{equation}    
  
Recall that $\A$ is invertible if and only if $\A$ is injective and  has a closed and dense range. By the Hahn-Banach theorem, $\A$ has dense range if and only if no non-zero continuous functional on  $\V'$ annihilates the range of $\A$. By reflexivity, this is equivalent to the following uniqueness property:

\begin{equation}\label{eq:3.6}
\mbox{ for all } v\in\V,\big( a(u,v)=0 \mbox{  for all }  u\in\U\big)\Rightarrow v=0. 
\end{equation}   
Thus (\ref{eq:3.1}) is \emph{well-posed} (i.e. for all $L\in\V'$ there exists a unique $u\in\U$ satisfying (\ref{eq:3.1})) if and only if (\ref{eq:3.5}) and  (\ref{eq:3.6}) are satisfied. In fact, Hadamard's definition of well-posedness also requires continuity of the inverse operator, which here  automatically follows from bijectivity by the Inverse Mapping Theorem.

In order to obtain a result of convergence of the approximate solutions we consider the following \emph{uniform Banach-Ne\v cas-Babu\v ska condition} (called \emph{Ladyzenskaia-Babu\v ska-Brezzi condition} in the framework of the mixed formulations,  i.e. approximation of saddle point problems, see also Section~\ref{sec:8}), which is the estimate (\ref{eq:3.5}) for 
$a_{|\U_n\times\V_n }$ uniformly in $n\in\N$, namely
\[(BNB)\quad \exists \beta>0; \forall n\in\N,~\forall u\in\U_n,~\sup_{v\in\V_n,\|v\|_\V=1}|a(u,v)|\geq \beta \|u\|_{\U }.        \]
 
\begin{rem}
	Condition (BNB) is also called the \emph{inf--sup condition} since by the Hahn-Banach Theorem it can be reformulated as 
	
\[ \exists \beta>0; \forall n\in\N, ~\inf_{u\in\U_n,\|u|_V=1} ~\sup_{v\in\V_n,\|v\|_\V=1}|a(u,v)| \ge \beta .        \]

More precisely, this is the \emph{uniform} or \emph{discrete }BNB-condition which is used for approximation whereas (\ref{eq:3.5}) is the \emph{continuous} BNB-condition which expresses well-posedness of the problem and can also be expressed by an inf-sup-condition  (see for example \cite[Lemma 6.95 and Lemma 6.110]{hackbusch}). The use of (LBB) relates this inequality to the work of Ladyzhenkaya \cite{ladyzhenskaya} who, after a previous contribution due to Babuska \cite{babuska}, used it to prove well-posedness. Brezzi \cite{brezzi} introduced the analogue of the uniform BNB-condition for the treatment of saddle point problems (see Section~\ref{sec:8} for more details).
	
Usually, in the numerical analysis community, one uses the name ``inf-sup'' condition (or LBB condition) only in the context of saddle point problems (see	condition \eqref{eq:infsupwzp}.$(iii)$ in Section~\ref{sec:8}). We keep the name ``(BNB) condition'',  following the monograph \cite{EG04}.
\end{rem}		
 We recall that (BNB) implies that the approximate solutions converge to the solution if the problem is well-posed (see for example \cite{EG04,XZ03, AU19}). Here we will show that (BNB) is actually equivalent to Galerkin-convergence, and surprisingly also to Galerkin-convergence   for the dual problem.
%Our main result in this section shows that this condition is equivalent to the convergence of the approximate solutions in the following sense.   
\begin{defn}[Convergence of Galerkin approximation]\label{def:3.1}
	We say that \emph{the Galerkin-approximation converges} if (\ref{eq:3.1}) as well as (\ref{eq:3.2}) are well-posed for all $n\in\N$ and $L\in \V'$ and if, in addition, there exists a constant $\gamma>0$ independent of $n$ and $L$ such that, 
	\begin{equation}\label{eq:3.7}   
	\|u-u_n\|_{\U}\leq \gamma \dist(u,\U_n),
	\end{equation} 
	where $u$ is the solution of (\ref{eq:3.1}) and $u_n$ the solution of (\ref{eq:3.2}) for $n\in\N$ and $L\in \V'$. In particular, $\lim_{n\to \infty} u_n=u$ in $\U$.  
\end{defn} 
 We may also consider the dual problem of (\ref{eq:3.1}) where $a$ is replaced by the \emph{adjoint form} $a^*:\V\times\U\to \K$ given by 
 \[a^*(v,u)=\overline{a(u,v)}\quad(u\in\U,v\in\V).\]
   If in Definition~\ref{def:3.1} the form $a$ is replaced by $a^*$, then we say that \emph{the dual Galerkin approximation converges}.  Similarly we note the following dual \emph{uniform Banach-Ne\v cas-Babu\v ska  condition}
\[(BNB^*)\quad \exists \beta^*>0; \forall n\in\N, \sup_{u\in\U_n,\|u\|_\U=1}|a^*(u,v)|\geq \beta^* \|v\|_{\V }\quad (v\in\V_n). \]
Then the following theorem holds. 
\begin{thm}\label{th:3.2}
	The following assertions are equivalent:
	\begin{enumerate}
	\item[(i)] the Galerkin approximation converges;
	\item[(ii)] $(BNB)$ holds;
	\item[(iii)] $(BNB^*)$ holds;
	\item[(iv)] the dual Galerkin approximation converges.  
	\end{enumerate}	  
\end{thm}
It is surprising that $(BNB)$ and $(BNB^*)$ are equivalent even though the corresponding condition (\ref{eq:3.5}) is obviously not equivalent to its dual form. In fact, it can well happen that $\A$ is injective and has closed range (so that there exists $\beta>0$ satisfying (\ref{eq:3.5})) but the range of $\A$ is a proper subspace of $\V'$ so that there exists $v\in\V$ such that $v\neq 0$ and $a(u,v)=0$ for all $u\in\U$; in particular the dual form of (\ref{eq:3.5}) does not hold for any $\beta^*>0$. 

We will give the proof of Theorem~\ref{th:3.2} in several steps which give partly even stronger results. At first we show that $(ii)$ implies $(i)$, where $\gamma$ can even be expressed in terms of $\beta$ and $M$. Although the proof of this result is classical (see for example \cite{XZ03,EG04}), we provide it for the convenience of the reader, but also to establish the well-posedness of (\ref{eq:3.1}) which we did not assume. This will be important for the proof of Theorem~\ref{th:3.2} and for the main result in Section~\ref{sec:4}. 
\begin{prop}\label{prop:3.3}
	Let $\beta>0$. Assume that for all $n\in\N$, 
	\begin{equation}\label{eq:bnb}
	\sup_{v\in\V_n,\|v\|_\V=1}|a(u,v)|\geq \beta \|u\|_{\U }\quad (u\in\U_n). 
	\end{equation}
	 Then the Galerkin-approximation converges and (\ref{eq:3.7}) holds with \[\gamma=1+\frac{M}{\beta}.\]  
\end{prop}
\begin{proof}
Let $L\in \V'$. Note that $(\ref{eq:bnb})$ implies (\ref{eq:3.3}). Thus, for each $n\in\N$ there exists a unique solution $u_n$ of (\ref{eq:3.2}).   By $(\ref{eq:bnb})$, 
\begin{equation}\label{eq:3.8prev}  
\|u_n\|_{\U}\leq \frac{1}{\beta} \sup_{  v\in\V_n,\|v\|_\V\leq 1} |\langle L,v\rangle |\leq \frac{1}{\beta} \|L\|_{\V'}.  
\end{equation}
Since $\U$ is reflexive, we find $u\in \U$ such that a subsequence of $(u_n)_n$, say, $(u_{n_k})_{k}$, converges weakly  to $u$. Let $v\in\V$. By assumption we find $v_k\in \V_{n_k}$ such that $\lim_{k\to\infty}\|v-v_{n_k}\|_\V=0$. It follows that 
\[ a(u,v) =\lim_{k\to\infty} a(u_{n_k}, v_k)=\lim_{k\to\infty}  \langle L,v_k\rangle=\langle L,v\rangle. \]
Thus we find a solution $u$ of (\ref{eq:3.1}). But so far we do not know its uniqueness. This will be a consequence  of $(\ref{eq:3.7})$ which we prove now.  Indeed, observe that  
\begin{equation}\label{eq:3.8}
a(u,\chi)=\langle L,\chi\rangle=a(u_n,\chi)\mbox{ for all }\chi\in \V_n.
\end{equation}
It follows that 
$ a(u,\chi)=a(u_n,\chi)\mbox{ for all }\chi\in\V_n$ (\emph{Galerkin orthogonality}).
Using this, for all   $w\in\U_n$,
\begin{eqnarray*}
\|u-u_n\|_\U & \leq & \|u-w\|_\U + \|w -u_n\|_{\U}\\
   & \leq & \|u-w\|_\U + \frac{1}{\beta} \sup_{v\in\V_n,\|v\|_\V=1}|a(w-u_n,v)|\\
    & = & \|u-w\|_\U + \frac{1}{\beta} \sup_{v\in\V_n,\|v\|_\V=1}|a(w-u,v)|\\
   & \leq & \left( 1+\frac{M}{\beta} \right)\|w -u\|_\U	.
\end{eqnarray*}	
Taking the infimum over all $w\in \U_n$ we obtain (\ref{eq:3.7}). In particular $\lim_{n\to \infty}\|u-u_n\|_\U =0$ which shows uniqueness. 
\end{proof}	

The following result is due to Xu and Zikatanov   \cite[Theorem~2]{XZ03} (see also \cite[Satz 9.41]{AU19}). We nevertheless provide its proof for the sake of completeness.
\begin{prop}\label{prop:3.3KatoArendt}
	Assume that $\U$ is a Hilbert space and that $\beta>0$ is such that  (\ref{eq:bnb}) holds.  
	Then the Galerkin-approximation converges and (\ref{eq:3.7}) holds with $\gamma=\frac{M}{\beta}$.  
\end{prop}
\begin{proof}
%We follow the proof of \cite[Theorem~2]{XZ03} (see also \cite[Satz 9.41]{AU19}). 
Note that 
(\ref{eq:bnb}) implies (\ref{eq:3.3}). Consequently for each $w\in\U$ there exists a unique $Q_nw\in\U_n$ such that 
\[    a(Q_n w,\chi)=a(w,\chi)\quad (\chi\in\V_n).\]
Then $Q_n$ is a projection from $\U$ onto $\U_n$, which is calle the \emph{Ritz projection.} Moreover, 
\begin{eqnarray*}
 \beta \|  Q_n w\|_{\U} & \leq & \sup_{\chi\in\V_n,\|\chi\|_\V=1} |a(Q_nw,\chi)|\\
  & = & \sup_{\chi\in\V_n,\|\chi\|_\V=1} |a(w,\chi)|\\
  & \leq & M\|w\|_{\U}.	
\end{eqnarray*} 	
Thus $\|Q_n\|\leq \frac{M}{\beta}$. 

Since $U_n\neq 0$ and $U\neq 0$, one has $Q_n\neq 0,\Id$. It follows from a result due to Kato \cite[Lemma 4]{kato1060est} that   $\Vert Q_n\Vert =\Vert {\rm Id} - Q_n\Vert$.

Now let $L\in\V'$ and $u$ the solution of (\ref{eq:3.3}), $u_n$ the solution of (\ref{eq:3.2}). Then for any $\chi\in\U_n$, 
\[u-u_n=(\Id -Q_n)u=(\Id -Q_n)(u-\chi).\]
Hence 
\[\|u-u_n\|_{\U} \leq \|\Id -Q_n\|\|u-\chi\|_\U=\|Q_n\| \|u-\chi\|_{\U}\leq \frac{M}{\beta}\|u-\chi\|_{\U}.\]
This implies that 
\[  \|u-u_n\|_{\U}\leq \frac{M}{\beta}\dist (u,\U_n).  \]
\end{proof}

\begin{rem}
	Also in certain Banach spaces an improvement of the constant $1+\frac{M}{\beta}$ is possible, see Stern \cite{S15}.
\end{rem}

Next we show that even a weaker assumption than the convergence of the Galerkin-approximation implies $(BNB^*)$. 
\begin{prop}\label{prop:3.4}
Assume (\ref{eq:3.3}) for all $n\in\N$ and  that \[\sup_{n\in\N}\|u_n\|_\U<\infty\] whenever $L\in\V'$ and $u_n$ is the solution of (\ref{eq:3.2}). Then $(BNB^*)$ holds. 
\end{prop}
\begin{proof}
	Since the spaces $\V_n$ and $\U_n$ have the same finite dimension, our assumption (\ref{eq:3.3}) implies also dual uniqueness, i.e. $a(\chi,v)=0$ for all $\chi\in \U_n$ implies $v=0$ whenever $v\in \V_n$, and this for all $n\in\N$. Thus 
	\[ \|v\|_{\V_n}:=\sup_{u\in\U_n,\|u\|_\U=1}|a(u,v)| \]
	defines a norm on $\V_n$.  Moreover, 
	\[ |a(u,v)|\leq \|u\|_\U \|v\|_{\V_n} \mbox{ for all }u\in\U_n,v\in\V_n.    \]
	We show that the set 
	  \[  {\mathcal B}:=\left\{  \frac{v}{\|v\|_{\V_n}}:n\in\N,v\in\V_n,v\neq 0      \right\}   \]
	  is bounded. For that purpose, let $L\in\V'$. By assumption there exist $c>0$ and $u_n\in\U_n$ such that 
	  \[a(u_n,v)=\langle L,v\rangle \mbox{ for all }v\in \V_n\]
	  and $\|u_n\|_\U\leq c$ for all $n\in\N$. Now, for $\frac{v}{\|v\|_{\V_n}}\in {\mathcal B}$, 
	  \[  \left|\langle L,\frac{v}{\|v\|_{\V_n}}\rangle \right|=|a(u_n,v)|\frac{1}{\|v\|_{\V_n}}\leq \|u_n\|_\U\leq c.\] 
	  This shows that $\mathcal B$ is weakly bounded and thus, owing to the Banach--Steinhaus theorem, norm-bounded. Therefore there exists $\beta^*>0$ such that
	  $ \|v\|_\V\leq \frac{1}{\beta^*}\|v\|_{\V_n}$, i.e. 
\[  \beta^* \|v\|_\V\leq \sup_{u\in\U_n,\|u\|_\U=1}|a(u,v)| \mbox{ for all }v\in\V_n,n\in\N.    \]
This is $(BNB^*)$. 
\end{proof}	
\begin{proof}[Proof of Theorem~\ref{th:3.2}]
$(ii)\Rightarrow (i)$ and $(iii)\Rightarrow (iv)$ via Proposition~\ref{prop:3.3}, whereas 
$(i)\Rightarrow (iii)$  and $(iv)\Rightarrow (ii)$ follows from Proposition~\ref{prop:3.4}.\\
 \end{proof}
\textbf{Remark:} The hypothesis on $\U$ and $\V$ to be  reflexive is not needed in Proposition~\ref{prop:3.3}.\\

Finally we mention that the best lower bounds $\beta$ for $(BNB)$ and $\beta^*$ for $(BNB^*)$ are the same if $\U$ and $\V$ are Hilbert spaces. 
\begin{prop} \label{prop:eqbetabetastar}  Assuming that $\U$ and $\V$ are Hilbert spaces, let $\beta>0$. Then the two conditions $(\ref{eq:11})$ and $(\ref{eq:12})$ are equivalent:
\begin{equation}\label{eq:11}
\sup_{  \|v\|_\V\leq 1, v\in\V_n} |a(u,v)|\geq \beta \|u\|_\U\mbox{ for all }u\in\U_n\mbox{ and for all }n\in\N;
\end{equation} 	
\begin{equation}\label{eq:12}
\sup_{  \|u\|_\U\leq 1, u\in\U_n} |a(u,v)|\geq \beta \|u\|_\U\mbox{ for all }v\in\V_n\mbox{ and for all }n\in\N;
\end{equation} 	
\end{prop}	
\begin{proof}
Let $n\in\N$ and $A_n:\U_n\to\V_n$ be given by 
\[ \langle A_n u,v\rangle_\V =a(u,v).   \]	
Then 
\[  \langle A_n^* v,u\rangle_{\U}=a^*(v,u)=\overline{a(u,v)},    \]
where $A_n^*$ is the adjoint of $A$. Moreover, since $A_n$ is invertible, 
\[  \sup_{  \|v\|_\V= 1, v\in\V_n} |a(u,v)|\geq \beta \|u\|_\U  \]
for all $u\in \U_n$ if and only if $\|A_n^{-1}\|\leq \frac{1}{\beta}$. Since $(A_n^*)^{-1}=(A_n^{-1})^*$, it follows that $\| (A_n^*)^{-1}\|= \|(A_n^{-1})^*\|=\|A_n^{-1}\|\leq \frac{1}{\beta}$ and hence 
\[  \sup_{  \|u\|_\U\leq 1, u\in\U_n}|a^* (v,u)|\geq \beta\|v\|_\V\mbox{ for all }v\in\V_n.\]
\end{proof}

		W. V. Petryshyn, namely  in Theorem 2  
		and 3 of \cite{petryshyn},  considers approximation of an operator equation by finite dimensional problems and characterizes strong convergence. However, besides in very special situations, it sems not possible to deduce from this convergence of a Galerkin approximation, formulated in terms of sesquilinear forms. Further results for operator equations and their approximation  can be  found in the monograph \cite[p. 26 ff]{Silbermann}.

\section{
	Existence of a converging Galerkin approximation}\label{sec:new3}
In this section, we again let $\U$ and $\V$ be separable reflexive real Banach spaces and let $a:\U\times \V\to \R$ be a continuous sesquilinear form such that the problem \eqref{eq:3.1}
%\begin{equation}\label{eq:3.1}
%a(u,v)=\langle L,v\rangle \quad (v\in \V)
%\end{equation}  
is well-posed; i.e. for all $L\in\V'$ there exists a unique $u\in\U$ satisfying (\ref{eq:3.1}).  Since $\U$ and $\V$ are separable, there always exist approximating sequences 
$(\U_n)_{n\in\N} $ of $\U$ and $(\V_n)_{n\in\N}$ of $\V$. Our question is whether there is a choice of these sequences which is adapted to the problem (\ref{eq:3.1}); i.e. such that 
the associated Galerkin approximation converges. We will show that the answer  is related to the approximation property. In fact, different versions of this property play a role; we recall them in the next definition.
\begin{defn}[Approximation property and Schauder decomposition]
 Let $\X$ be a separable Banach space. 
\begin{itemize}
        \item[a)] The space $\X$ has the \emph{approximation property} (AP) if, for every compact subset $K$ of $\X$ and every $\varepsilon>0$, there exists a finite rank operator $R\in{\mathcal L}(\X)$ such that 
        \[    \|    Rx-x\|<\varepsilon \mbox{ for all }x\in K.   \] 
    \item[b)] The space $\X$ has the \emph{bounded approximation property} (BAP) if there exists a sequence $(P_n)_{n\in\N}$ of finite rank operators in $\X$ such that 
	\[\hbox{for all }x\in \X,~\lim_{n\to \infty}   P_n x=x.\]
	\item[c)] The space $\X$ has the \emph{bounded projection approximation property} (BPAP) if each $P_n$ in b) can be chosen as a projection (i.e. such that $P_n^2=P_n$).
%	\item[c)] The space $\X$ has the \emph{increasing bounded projection approximation} (IBPAP) if in addition of b) the $P_n$ can be chosen such that 
%	$P_n \X\subset P_{n+1}\X$.      
	 \item[d)] The space $\X$ possesses a \emph{finite dimensional 
	 	 decomposition} if one finds $(P_n)_{n\in\N}$ as in c) with the additional property
	 \begin{equation}\label{neq:3.2}
	 P_mP_n=P_nP_m=P_m\mbox{ for all }n\geq m.
	 \end{equation}
	 \item[e)] The space $\X$ has a \emph{Schauder basis} if d) holds with \[\dim (P_n - P_{n-1})\X =1\mbox{ for all }n\in\N.\] 
\end{itemize}
\end{defn}   
It is known that (BAP) is equivalent to (AP) if $\X$ is reflexive. The first counterexample of  a Banach space without (AP) has been given by Enflo \cite{enflo}. He constructed a space which is even separable and reflexive.

Obviously the properties a)--e) have decreasing generality. It was Read \cite{read86} who showed that (BAP) does not imply (BPAP), even if reflexive and separable spaces are considered. Szarek \cite{szarek} constructed a reflexive, separable Banach space having a finite dimensional Schauder decompositon but not a Schauder basis. Finally, it seems to be unknown whether (BPAP) implies the existence of a finite dimensional Schauder decomposition (see \cite[Sec. 5.7.4.6]{Pi07} and \cite[Problem 6.2]{Cas}). However,  if $\X$ is reflexive and separable, then these two properties are equivalent by \cite[Theorem 6.4 (3)]{Cas}). 

Concerning the notion of finite dimensional Schauder decomposition, there is an equivalent formulation, namely the existence of  finite dimensional subspaces $\X_n$ of $\X$  such that for each $x\in\X$ there exist unique $x_n\in\X_n$ such that $x=\sum_{n\in\N}x_n$
% (choose in c) $\X_1=P_1\X$, $\X_n=(P_n-P_{n-1})(\X)$ for $n\geq 2$). 
This explains the name. We refer to \cite[Chapter I]{LT77} , \cite{Cas} for more information and to \cite[Sec. 5.7.4]{Pi07} for the history of the approximation property.  In the following theorem, by the  hypothesis of well-posedness, the two Banach spaces $\U$ and $\V$ are isomorphic. For this reason they have the same Banach space properties.
 
\begin{thm}\label{th:new3.2}
	Let $\U$ and $\V$ be separable reflexive Banach spaces and let $a:\U\times\V\to \K$ be a continuous sesquilinear form such that (\ref{eq:3.1}) is well-posed.  Then the following assertions are equivalent.
	\begin{itemize}
		\item[(i)]  There exist approximating sequences  $(\U_n)_{n\in\N}$ of \,$\U$ and  $(\V_n)_{n\in\N}$ of \,$\V$ such that the associated Galerkin approximation converges.
		\item[(ii)] The space \,$\U$ has the (BPAP).
		\item[(iii)] The space \,$\U$ has a finite dimensional Schauder decomposition.
	\end{itemize} 
\end{thm}
Here convergence of the associated Galerkin approximation is understood in the sense of Definition~\ref{def:3.1}.   
\begin{proof}[Proof of Theorem~\ref{th:new3.2}]
$(i) \Rightarrow (ii)$ Let $u\in\V$. Then $\langle L,v\rangle :=a(u,v)$ defines an element $L\in\V'$. By Definition~\ref{def:3.1}, for each $n\in \N$, there exists  a unique $P_n u\in\V_n$ such that 
\[  a(P_n u,\chi) =a(u,\chi)\mbox{ for all }\chi\in\V_n .  \]
Moreover, $\|P_nu-u\|\leq \gamma \dist(\U_n,u)$
for all $n\in\N$ and some $\gamma>0$. In particular, $\lim_{n\to \infty}P_nu=u$. It follows from the definition that $P_n^2=P_n$. Since $P_n\U\subset \U_n$, each $P_n$ has finite rank. We have shown that the space $\U$ has the (BPAP). \\
$(ii) \Rightarrow (iii)$ See \cite[Theorem 6.4 (3)]{Cas}.\\
$(iii) \Rightarrow (i)$ Let $\A :\U\to\V'$ be the operator defined by $\langle \A u,v\rangle =a(u,v)$. Then $\A$ is invertible. By hypothesis there exist finite rank projections $(P_n)_{n\in\N}$ such that $\lim_{n\to \infty}P_nu=u$ for all $u\in\U$. Let $L\in\V'$, $u:=\A^{-1}L$ be the solution of (\ref{eq:3.1}). Then
 \begin{equation}\label{neweq:3.3}
 u_n:=P_n \A^{-1} L\to u\mbox{ in }\U\mbox{ as }n\to\infty.
 \end{equation}   	
 We show that $u_n$ is obtained as a Galerkin approximation. In fact, fix $n\in\N$. There exist $b_1,\cdots, b_m\in\U$, $\varphi_1,\cdots,\varphi_m\in\U'$ such that  $\langle \varphi_i,b_j\rangle =\delta_{i,j}$ and 
 \begin{equation}\label{eq:new3.4}
  P_n x=\sum_{k=1}^m\langle \varphi_k ,x\rangle b_k  
\end{equation}    
for all $x\in\U$. Since $\V$ is reflexive there exist $v_k\in\V$ such that 
\begin{equation}\label{eq:new3.5}
\langle \varphi_k,\A^{-1}g\rangle =\langle g,v_k\rangle
\end{equation} 
for all $g\in\V'$ and $k=1,\cdots,m$. Define $\V_n=\span \{v_1,\cdots , v_m\}$ and $\U_n=\span\{b_1,\cdots,b_m   \}$. Now consider the given $L\in \V'$. Let $w=\sum_{k=1}^m \lambda_k b_k\in\U_n$. Then 
\begin{equation}\label{eq:new3.6}
a(w,\chi)=\langle L,\chi\rangle \mbox{ for all }\chi\in \V_n
\end{equation}
if and only if 
\begin{equation}\label{eq:new3.7}
a(w,v_j)=\langle L,v_j\rangle \mbox{ for }j=1,\cdots,m.
\end{equation}
By (\ref{eq:new3.5}), 
\[ a(w,v_j) =\sum_{k=1}^m \lambda_k a(b_k,v_j)= \sum_{k=1}^m \lambda_k\langle \A b_k,v_j\rangle= \sum_{k=1}^m \lambda_k \langle \varphi_j, b_k\rangle = \lambda_j. \]   
Therefore $w=\sum_{k=1}^m\langle L,v_k\rangle b_k$ is the unique solution of (\ref{eq:new3.6}). Again, by (\ref{eq:new3.5}), 
\[ u_n=P_n \A^{-1} L=\sum_{k=1}^m \langle \varphi_k, \A^{-1}L\rangle b_k= \sum_{k=1}^m \langle L,v_k\rangle b_k=w,   \]
and it follows from (\ref{neweq:3.3}) that $\lim_{n\to \infty}u_n=u$. This also implies that $\dist(\U_n,u)\to 0$ as $n\to\infty$. Thus the sequence $(\U_n)_{n\in\N}$ is approximating. 

It remains to show that  the sequence $(\V_n)_{n\in\N}$ is approximating in $\V$. For this we need the the additional property (\ref{neq:3.2}). Consider the adjoint $P_n'\in{\mathcal L}( \U')$ of $P_n$. Then $P_n'\varphi$ weakly converges to $\varphi$ as $n\to\infty$ for all $\varphi\in\U'$. Thus 
\[\W:=   \cup_{n\in\N}P_n'\U'\]
is weakly  dense in $\U'$. But, because of (\ref{neq:3.2}), $\W$ is a subspace of $\U'$. Thus, by Mazur's Theorem, $\W$ is dense in $\U'$.  If $\psi\in \W$, then there exist $m\in\N,\varphi\in\U'$ such that $\psi=P_m'\varphi$. Thus 
\[P_n'\psi=P_n'P_m'\varphi  =P_m'\varphi=\psi,\]
for all $n\in\N$ by (\ref{neq:3.2}), and then $\lim_{n\to \infty}P_n'\psi=\psi$ for all $\psi\in \W$. Since $\sup_{n\in\N}\|P_n'\|<\infty$, it follows that $\lim_{n\to \infty}P_n'\varphi=\varphi$ for all $\varphi\in \U'$. This implies that the sequence $(P_n'\U')_{n\in\N}$ is approximating in $\U'$. It follows from (\ref{eq:new3.5}) that $\V_n\supset(\A^{-1})'P_n'\U'$. In fact, fix $n$ and consider $P_n$ as in (\ref{eq:new3.4}). Then (\ref{eq:new3.5}) says that $v_k=(\A^{-1})'\varphi_k$. Since $(P_n'\U')_{n\in\N}$ is an approximating sequence in $\U'$ and  $(\A^{-1})'$ is an isomorphism from $\U'$ to $\V$, it follows that $(\V_n)_{n\in\N}$ is an approximating sequence in $\V$. 
\end{proof}	
\section{Essentially coercive forms}

Let $\V$ be a separable Hilbert space over $\K=\C$ or $\R$ and  $a:\V\times \V\to \K$ be a sesquilinear form satisfying 
\[ |a(u,v)|\leq M\|u\|_{\V} \|v\|_{\V}\mbox{ for all }u,v\in \V \]
for some $M>0$.   %Such sesquilinear forms are called \emph{continuous}. 
Then we may associate with $a$ the operator $\A\in {\mathcal L}(\V,\V')$ defined by 
\[   \langle \A u,v\rangle =a(u,v).\]
If $a$ is \emph{coercive}, i.e. if 
\[|a(u,u)|\geq \alpha \|u\|_{\V}^2   \quad (u\in \V)\]
for some $\alpha>0$, then $\A$ is invertible. This consequence is the well-known Lax-Milgram lemma. 

\begin{rem}
			The notion of \emph{coercivity} is not uniform in the literature. Ours is the natural hypothesis of the Lax-Milgram Lemma and is conform with the Wikipedia  entry "Babuska-Lax-MilgramTheorem". In non-linear analysis there is a wide agreement on this notion: In the real case, a possibly non-linear operator $\A\in {\mathcal L}(\V,\V')$ is called \emph{coercive} if  there exists a function $\eta : \R \to \R$ such that $\eta(t) \to \infty$ as $t \to \infty$ and $\langle \A v,v\rangle \geq \eta(\|u\|_\V) \|v\|_\V $ for all $v\in \V$. If $\A$ is linear this is equivalent to the existence of $\alpha>0$ such that \[\langle \A u,u\rangle\geq \alpha \|u\|_{\V}^2   \quad (u\in \V),\]
i.e. our condition without the absolute value. This is a "forcing condition" which justifies the name \emph{coercive}. Other authors prefer the word \emph{$\V-$ellipticity}, see e.g. \cite{hackbusch}, \cite{lions-mangenes}. We use \emph{elliptic} for \emph{shifted coercivity} in \cite{AtKS14}, see also the remark at the end of this section.

	\end{rem}

Our aim is to find weaker assumptions than coercivity which help to decide whether the operator $\A$ is invertible. 

Note that $a$ is coercive if and only if 
\[  \lim_{n\to \infty} a(u_n,u_n)=0\mbox{ implies that }\lim_{n\to \infty}\|u_n\|_{\V}=0. \] 
We weaken this property in the following way.
\begin{defn}[Essential coercivity]\label{def:2.1}
	The continuous sesquilinear form $a$ (or the operator $\A$) is called essentially coercive if 
	for each sequence $(u_n)_{n\in \N}$ in $\V$ weakly converging to $0$ and such that $\lim_{n\to \infty}a(u_n,u_n)=0$, one has $\lim_{n\to \infty}\|u_n\|_{\V}=0$. 
\end{defn}
The following is a characterization of this new property. 
\begin{thm}\label{th:new2}
	The following assertions are equivalent:
	\begin{enumerate}
		\item[(i)]  the form $a$ is essentially coercive;
		\item[(ii)] there exist an orthogonal projection $P\in{\mathcal L}(\V)$ of finite rank and $\alpha>0$ such that 
		\[   |a(u,u)| +\| Pu\|_\V^2 \geq \alpha \| u\|_\V^2\mbox{ for all }u\in \V;   \]
		\item[(iii)] there exist a Hilbert space $\GH$, a compact operator $J:\V\to \GH$ and $\alpha>0$ such that 
		\[ |a(u,u)|+\|   Ju\|_{\GH}^2 \geq \alpha \|u\|_{\V}^2\quad (u\in \V);  \]
		\item[(iv)] there exist a compact operator ${\mathcal K}\in {\mathcal L}(\V,\V')$ and $\alpha>0$  such that 
		\[ |a(u,u)|+|\langle {\mathcal K}u,u\rangle| \geq \alpha \|u\|^2_\V\quad (u\in\V).  \]
		%$A+K$ is coercive.
	\end{enumerate}
\end{thm}
\begin{proof}
	$(i)\Rightarrow (ii)$: Let $(e_n)_{n\in \N}$ be an orthonormal basis of $\V$ and consider the orthogonal projections  $P_n$ given by 
	\[   P_n v:=\sum_{k=1}^{n} \langle v ,e_k\rangle_\V \;e_k.     \]
	
	Assume that (ii) is false for every $P_n$. 
	 Then there exists a sequence  $(u_n)_{n\in\N}\subset  \V$ such that  $\|u_n\|_\V=1$ and \[|a(u_n,u_n)| +\|P_nu_n\|_\V^2<\frac{1}{n}.\] Note that, since $\Id-P_n$ is a self-adjoint operator,
	\[ |\langle (\Id -P_n)u_n,v\rangle_\V|=|\langle u_n,(\Id-P_n)v\rangle_\V|\leq \|(\Id-P_n)v\|_\V,  \]
	with $\lim_{n\to \infty}    \|(\Id-P_n)v\|_\V=0$ for all $v\in \V$. This implies that $(\Id-P_n)u_n$ converges weakly to $0$. 	Since $\lim_{n\to\infty}\|P_n u_n\|_\V=0$, it follows that $u_n$ converges weakly to $0$. Moreover $\lim_{n\to\infty}|a(u_n,u_n)|\leq \lim_{n\to \infty}\frac{1}{n}=0$. Therefore $a$ is not essentially coercive.  \\
	$(ii)\Rightarrow (iii)$: Choose $\GH=\V$ and $J=P$. \\
	$(iii)\Rightarrow (iv)$: There exists a unique operator $J^*:\GH\to \V'$ such that 
	\[  \langle J^* u,v\rangle =\langle u,Jv\rangle_{\GH}  \] 
	for all $v\in\V$. 
	Choose ${\mathcal K}=J^*J$. \\
	$(iv)\Rightarrow (i)$: Let $(u_n)_{n\in\N}\subset \V$ that tends weakly to $0$ and such that $a(u_n,u_n)=\langle \A u_n,u_n\rangle$ tends to $0$ as $n\to\infty$. Since ${\mathcal K}$ is compact, $\|{\mathcal K}u_n\|_\V\to 0$ as $n\to\infty$. Hence $|\langle {\mathcal K}u_n,u_n\rangle_\V|\to 0$ as $n\to\infty$.   By assumption   there 
	exists $\beta>0$ such that 
	\[ |\langle \A u_n,u_n\rangle |+|\langle {\mathcal K}u_n,u_n\rangle|\geq \beta \|u_n\|_{\V}^2.   \] 
	 It follows 
	that $\|u_n\|_\V\to 0$ as $n\to\infty$.  
\end{proof}
Next we want to justify the notion "essentially coercive". We recall that by the Toeplitz--Hausdorff theorem \cite{gusta}, the \emph{numerical range} of $a$,
\[W(a):=\{    a(u,u):u\in\V,\|u\|_\V =1 \},\] 
is a convex set. Hence also $\overline{W(a)}$ is convex. For $\alpha>0$,
\[|a(u,u)|\geq \alpha \|u\|^2_\V  \;\;\;\; (u\in\V)  \]  
if and only if 
\[  \overline{W(A)}\cap D_\alpha=\emptyset,  \]
where $D_\alpha=(-\alpha,\alpha)$ in the real case and $D_\alpha =\{ w\in\C : |w|<\alpha  \}$ if $\K=\C$. This observation leads to the following more  
precise description of coercivity.  
\begin{lem}\label{lem:new-4.3}
	The form $a$ is coercive if and only if there exist $\alpha>0$ and $\lambda\in\K$ with $|\lambda|=1$ such that  
	 \[ \mathop{Re}(\lambda z)\geq \alpha\mbox{ for all } z\in W(a). \]
\end{lem} 
\begin{proof}
We give the proof for $\K=\C$. Assume that $a$ is coercive. There exists a maximal $\alpha>0$ such that 
$\overline{W(a)}\cap D_a=\emptyset$. Then there exists $z_0\in \overline{W(a)}$ of modulus $\alpha$; i.e. $z_0=e^{i\theta}\alpha$ for some $\theta\in\R$. The set $C:=e^{-i\theta}\overline{W(a)}$ is convex and closed. Moreover $\alpha\in C$ and $D_\alpha\cap C =\emptyset$. This implies that $\mathop{Re}(z)\geq \alpha$ for all $z\in C$. Indeed, let $z\in C$ such that $\mathop{Re}(z)<\alpha$. Then the segment $[\alpha,z]$ has a non-empty intersection with $D_\alpha$. Since $C$ is convex it follows that $z\not\in C$.  

Conversely, clearly,  if  there exists $\alpha>0$ such that $ \mathop{Re}(\lambda z)\geq \alpha$ for all $z\in W(a)$, then $a$ is coercive. 
\end{proof}
\begin{thm}\label{th:new4.4}
	Let $\A \in {\mathcal L}(\V,\V')$. The following assertions are equivalent:
\begin{enumerate}
	\item[(i)] the operator $\A$ is essentially coercive; 
	\item[(ii)] there exists a finite rank operator ${\mathcal K}:\V\to\V'$ such that $\A+{\mathcal K}$ is coercive;
	\item[(iii)] there exists a compact operator ${\mathcal K}:\V\to\V'$ such that $\A+{\mathcal K}$ is coercive.
\end{enumerate}	
\end{thm}
\begin{proof}
$(i)\Rightarrow (ii)$: Choose the orthogonal finite rank projection $P$ on $\V$ and $\alpha>0$ as in Theorem~\ref{th:new2} (ii). Let $\V_1=\ker V$ and $\V_2=\mathop{range} P$. Then $\dim \V_2<\infty$ and $|a(u,u)|\geq \alpha \|u\|^2_\V$ for all $u\in\V_1$. Let $j:\V\to\V'$ be the Riesz isomorphism  given by 
\[   \langle j(u),v\rangle  =\langle u, v\rangle_{\V}.   \]
Let $A=j^{-1}\circ \A\in{\mathcal L}(\V)$. Then $a(u,v)=\langle Au,v\rangle_{\V}$ for all $u,v\in\V$. Moreover $A$ has a matrix decomposition
   	\[ A=\left(\begin{array}{cc}
   	A_{11} & A_{12}\\
   	A_{21} & A_{22}
   	\end{array}\right)    \] 
according to the decomposition $\V=\V_1\oplus \V_2$ of $\V$. Since  $P$ is orthogonal,  $A_{11}$ is coercive. Thus,  by Lemma~\ref{lem:new-4.3}, there exists $z_0\in\C$ such that $|z_0|=1$ and 
\[   \mathop{Re}z_0 \langle A_{11}u,u\rangle\geq \alpha \|u\|_\V^2  \]
for all $u\in\V_1$. Since $\dim \V_2<\infty$, there   exists a finite rank operator $K_1\in{\mathcal L}(\V)$ such that 
\[    A+K_1 = \left(\begin{array}{cc}
A_{11} & 0\\
0 & 0
\end{array}\right). \] 	
Choose a further finite rank perturbation $K_2$ such that 
\[  B:=A+K_1+K_2 = 
\left(\begin{array}{cc}
A_{11} & 0\\
0 & \alpha \overline{z_0}\Id_{\V_2}
\end{array}\right)  . 
\]
Since $P$ is orthogonal, for $Q=\Id -P$, we get
\[ \langle Bu,u\rangle_{\V} = \langle A_{11}Qu,Qu\rangle_{\V}+ \alpha \overline{z_0}\langle Pu,Pu\rangle_{\V}.       \]
Hence 
\[   \mathop{Re}\langle z_0 Bu,u\rangle_{\V} \geq \alpha \|Qu\|^2_{\V}+\alpha \|Pu\|^2_{\V}=\alpha \|u\|^2_{\V}.   \]	
Now let ${\mathcal K}=j\circ (K_1+K_2)$. Then $\A+{\mathcal K}$ is coercive.\\
$(ii)\Rightarrow (iii)$ is obvious.\\
$(iii)\Rightarrow (i)$: Condition $(iii)$ implies clearly Condition $(iv)$ of Theorem~\ref{th:new2}; thus the claim $(i)$ follows from that theorem. 
\end{proof}

\begin{cor}\label{cor:2.3}
	Let $a$ be a continuous  essentially coercive sesquilinear form. The following assertions are equivalent:
	\begin{enumerate}
		\item[(i)] for all $L\in \V'$ there exists a unique $u\in \V$ such that 
		\[ a(u,v)=\langle L,v\rangle\mbox{ for all }v\in \V;  \]
		\item[(ii)] $a(u,v)=0$ for all $v\in \V$ implies that $u=0$ (\emph{uniqueness});
		\item[(iii)] for all $L\in \V'$ there exists $u\in \V$ such that  $a(u,v)=\langle L,v\rangle$ for all $v\in \V$ (\emph{existence}). 	
	\end{enumerate}
\end{cor}
\begin{proof}
	The assertion (i) means that $\A$ is invertible, the assertion (ii) means that $\A$ is injective and the assertion (iii) means that $\A$ is surjective.   By Theorem~\ref{th:new4.4}, there exists a compact operator ${\mathcal K}\in{\mathcal L}(\V,\V')$ such that $\A+{\mathcal K}=:{\mathcal B}$ is invertible. \\
	$(ii)\Rightarrow (i)$: Assume that $\A$ is injective. Write 
	\[\A=\B-{\mathcal K}=\B(\Id -\B^{-1}{\mathcal K}).\]
	Then also $(\Id -\B^{-1}{\mathcal K})$ is injective. Since $\B^{-1}{\mathcal K}$ is compact, it follows from the classical Fredholm alternative that $(\Id-\B^{-1}{\mathcal K})$ is invertible. Consequently also $\A$ is invertible. \\
	$(iii)\Rightarrow (i)$: If $\A$ is surjective, write $\A=(\Id-{\mathcal K}\B^{-1})\B$ to conclude that $(\Id-{\mathcal K}\B^{-1})$ is surjective. Again we deduce that $(\Id-{\mathcal K}\B^{-1})$ is invertible and so is $\A$.      	
\end{proof}

\begin{rem}
In the previous corollary we deduced from Theorem \ref{th:new4.4}   the Fredholm alternative. This conclusion is well-known, if a compact perturbation is given, see for example \cite[Theorem 22.D]{Zeidler},  or \cite[Lemma 6.108]{hackbusch}. Our point is that a priori it is not at all clear that the topological condition defining essential coercivity implies that the form is  a compact perturbation of a coercive form. This is what Theorem \ref{th:new4.4} shows. Note that, in \cite[p229]{Petryshyn75}, our notion of essential coercivity is attributed, under the name ``condition (S)'', to Felix Browder  \cite{browder}  if we identify the operator with a form.

\end{rem}

Moreover, we deduce from Theorem~\ref{th:new4.4} the following properties of essential coercivity.   
\begin{cor}
	\begin{enumerate}
		\item[(a)] The set of all essentially coercive operators on $\V$ is open in ${\mathcal L}(\V,\V')$. 
		\item[(b)] If $\A\in {\mathcal L}(\V,\V')$ is essentially coercive and ${\mathcal K}\in {\mathcal L}(\V,\V')$ is compact, then $\A+{\mathcal K}$ is essentially coercive. 
		\item[(c)] If $\A\in {\mathcal L}(\V,\V')$ is essentially coercive, then $\A$ is a Fredholm operator of index $0$.     
	\end{enumerate}	
\end{cor}	
The following example shows that the invertibility of $\A$  does not imply the essential coercivity of $a$. 
\begin{exam}
	Let $\V=\ell^2(\N)$, $\K=\R$ and 
	\[a(u,v)=\sum_{n=0}^\infty (-1)^n u_n v_n.\]  	
	Let $j$ be the Riesz isomorphism introduced in the proof of Theorem~\ref{th:new4.4}. Then $A:=j^{-1}\circ \A$ is a diagonal operator with merely $1$ and $-1$ in the diagonal. Thus $A$ and obviously $\A$ are clearly invertible. Let 
	$f_n=(0,\cdots, 1,1,0,\cdots)$ where the $1$ is a coordinate for $k=2n$ and $k=2n+1$. Then $\|f_n\|=\sqrt{2}$ and $(f_n)_n$ tends weakly to $0$ as $n\to\infty$. Moreover $a(f_n,f_n)=0$ for all $n$, which shows that $a$ is not essentially coercive. 
\end{exam}	
\begin{rem}
 	Let $\K=\C$.  In \cite{AtKS14} a continuous sesquilinear form $a$ is called \emph{compactly elliptic} if there exists a compact operator $J:\V\to \GH$, where $\GH$ is some Hilbert space and there exists $\alpha>0$ such that 
		\[  \mathop{Re} a(u,u) +\| Ju\|^2_{\GH}\geq \alpha \|u\|_\V^2.  \]
 In view of Theorem~\ref{th:new2}, each compactly elliptic form is essentially coercive. In fact the following holds: the form $a$ is essentially coercive if and only if there exists $\lambda\in \C\setminus\{0\}$ such that $\lambda a$ is compactly elliptic. 
\begin{proof}
If $\lambda a$ is compactly elliptic, then $\lambda a$ is essentially coercive and hence also $a$ is essentially coercive. Conversely, let $a$ be essentially coercive. By Theorem~\ref{th:new4.4}, there exists a compact operator ${\mathcal K}:\V\to\V'$ such that the form $b$ defined by 
\[  b(u,v)=a(u,v)+\langle {\mathcal K} u,v\rangle \]
is coercive. By Lemma~\ref{lem:new-4.3} there exist $\lambda\in\C$ of modulus one and $\alpha>0$ such that $\mathop{Re}(\lambda b(u,u))\geq \alpha \|u\|^2_{\V}$ for all $u\in\V$. Now let $j:\V\to \V'$ be the Riesz isomorphism. Then $J:=j^{-1}\circ {\mathcal K}:\V\to\V$ is compact. Choosing $\GH=\V$ we see that  $\lambda b$ is compactly elliptic. It follows from \cite[Proposition~4.4 (b)]{AtKS14} that $\lambda a $ is compactly elliptic. 	
\end{proof}	
  		
\end{rem}

\section{Characterization of the universal Galerkin property}\label{sec:4}

In this section we want to characterize those forms on a Hilbert space for which every Galerkin approximation converges, whatever be the choice of the approximating sequence. 
 
Let $\V$ be a separable,  infinite dimensional separable Hilbert space over $\K=\R$ or $\C$, and let 
 $a:\V\times\V\to \K$ be a continuous sesquilinear form.
 Given $L\in\V'$ we again consider solutions of the problem: 
 \begin{equation}\label{eq:pb5.1}
 \hbox{Find }u\in\V,~a(u,v)=\langle L,v\rangle \mbox{ for all }v\in\V. 
 \end{equation}
 We say that the form $a$ satisfies \emph{uniqueness} if for $u\in\V$, 
 \[  a(u,v)=0 \mbox{ for all }v\in\V\mbox{ implies }u=0. \] 
 We say that (\ref{eq:pb5.1}) is \emph{well-posed} if for all $L\in\V' $ there exists a unique solution $u\in\V$. 
\begin{defn}[Universal Galerkin property]
The sesquilinear and continuous form $a$ has the \emph{universal Galerkin property} if (\ref{eq:pb5.1}) is well-posed and the following holds. Let 
$(\V_n)_{n\in\N}$ be an arbitrary approximating sequence of $\V$ . Then there exist $n_0\in\N$ and $\gamma>0$ such that for each $L\in\V'$ and each $n\geq n_0$, there exists a unique $u_n\in\V_n$ solving 
\[ a(u_n,\chi)=\langle L,\chi\rangle \mbox{ for all }\chi\in\V_n,  \]
and
\[  \|u-u_n\|_{\V} \leq \gamma \dist (u,\V_n)  \mbox{ for all }n\geq n_0, \]
where $u$ is the solution of (\ref{eq:pb5.1}). 
\end{defn}

As recalled in the introduction and in the preceding section, the Lax-Milgram Theorem and C\'ea's Lemma imply the universal Galerkin property if $a$ is coercive. We now show that the weaker notion of essential coercivity also provides a sufficient condition for ensuring the universal Galerkin property, and moreover that it is necessary.   
   
\begin{thm}\label{th:4.1}
The following assertions are equivalent. 
\begin{enumerate}
	\item[(i)]   The form $a$ is essentially coercive and satisfies uniqueness.  
	\item[(ii)] The form $a$ has the universal Galerkin property. 
\end{enumerate}
\end{thm}
\begin{proof}
$(i)\Rightarrow (ii)$: let $(\V_n)_{n\in\N}$ be an approximating sequence in $\V$. 	By Theorem~\ref{th:3.2} it suffices to show that  there exist $\beta>0$ and $n_0\in\N$ such that 
\begin{equation}\label{eq:4.2}
\sup_{v\in\V_n,\|v\|_\V=1}|a(u,v)|\geq \beta\|u\|_\V \mbox{ for all }u\in\V_n,n\geq n_0.
\end{equation} 
Assume that (\ref{eq:4.2}) is false. We then find a subsequence $(n_k)_{k\in\N}$ and $u_{n_k}\in\V_{n_k}$ such that $\|u_{n_k}\|_\V=1$ and \[\sup_{v\in\V_{n_k},\|v\|_\V=1}|a(u_{n_k},v)|<\frac{1}{k} \mbox{ for all }k\in\N.\] 
We may assume that $(u_{n_k})_k$ converges weakly to $u$ taking a further subsequence otherwise. Let $v\in\V$. Then there exist $v_k\in\V_{n_k}$ such that $\lim_{k\to\infty}\|v-v_k\|_{\V}=0$. Thus
 \[ a(u,v)=\lim_{k\to\infty}a(u_{n_k},v_k)=0.\]
 It follows from the uniqueness assumption that $u=0$. Thus $(u_{n_k})_k$  converges weakly to $0$, $\lim_{k\to\infty}a(u_{n_k},u_{n_k}) =0$, but $\|u_{n_k}\|_\V=1$ for all $k$. Therefore the form $a$ is not essentially coercive. \\
 $(ii)\Rightarrow (i)$: the uniqueness condition is part of $(ii)$. It remains to show that $a$ is essentially coercive.   Let $(e_n)_{n\in\N}$ be an orthonormal basis of $\V$ and $\V_n:=\span \{e_1,\cdots,e_n\}$. By our assumption, there exist $2\leq n_0\in\N$ and  for all $n\geq n_0$ an operator $Q_n:\V\to\V_n$ such that
 \[ a(Q_nu,\chi)=a(u,\chi)\mbox{ for all }\chi\in\V_n\quad (n\geq n_0).    \]    
 Denote by $P_n:\V\to \V_n$ the orthogonal projection. Define the operator 
 \[  J_n:\V\to \V\times\V  \]
 by  
 \[ J_nu=(P_nu,Q_nu),\quad n\geq n_0.   \]
 Now assume that  $a$ is not essentially coercive. 
 Then it follows from Theorem~\ref{th:new2} that 
  for all $n\geq n_0$ we find $u_n\in\V$ such that $\|u_n\|_\V=1$ and 
 \[ |a(u_n,u_n)|+\|P_nu_n\|_{\V}^2 +\|Q_nu_n\|_\V^2<\frac{1}{(n+2)^2}.    \]
 In particular $\|P_n u_n\|_\V<\frac{1}{(n+2)^2}$. This implies that $u_n\not\in {\V_n}$. Let $\tilde{\V}_n=\span\{  \V_n\cup\{u_n\}  \}$. Then $(\V_n)_{n\geq n_0}$ and $(\tilde{\V}_n)_{n\geq n_0}$ are both approximating sequences.
 Let $n\geq n_0$ and let  $v\in \tilde{\V}_n$ be arbitrary with unit norm. There exist a unique $w_1\in\V_n$ and $\lambda\in\K$ such that 
 \[  v=w_1+\lambda u_n=w+\lambda(u_n-P_nu_n),    \]
 where $w:=w_1+\lambda P_n u_n\in\V_n$. Thus 
 \[  1=\|v\|^2_\V=\|w\|_\V^2 +|\lambda|^2\|u_n-P_n u_n\|^2_\V.    \]  
 Consequently $\|w\|^2_\V\leq 1$ and, since  $\|P_nu_n\|_{\V}<\frac{1}{2}$, it follows that   \[\|u_n-P_nu_n\|_\V\geq \frac{1}{2}, \] 
 which implies that  $|\lambda|^2\leq 4$,  i.e.  $|\lambda|\leq 2$. 
 
 Observe that the definition of $Q_n$ implies that  $a(u_n-Q_nu_n,w)=0$. Hence
 \begin{eqnarray*}
 |a(u_n,v)| & = & |a(u_n,w)+\lambda a(u_n,u_n-P_nu_n)|\\
  & = & |a(u_n-Q_nu_n,w)+a(Q_nu_n,w)+\lambda a(u_n,u_n-P_nu_n)|\\
  & \leq & |a(Q_nu_n,w)|+2|a(u_n,u_n)|+2|a(u_n,P_nu_n)|\\
   & \leq & \frac{M}{n+2}+\frac{2}{(n+2)^2}+\frac{2M}{(n+2)^2}.
  \end{eqnarray*}
Consequently 
\[  \lim_{n\to \infty}\sup_{v\in\tilde{V}_n,\|v\|_\V=1}  |a(u_n,v)|   =0. \]
Thus $(BNB)$ is violated for the approximating sequence $(\tilde{\V}_n)_{n\geq n_0}$. But then $(ii)$ does not hold by Theorem~\ref{th:3.2},
which shows that the assumption that $a$ is not essentially coercive is false. 

\end{proof}	
It is obvious that a form $a$ is essentially coercive if and only if its adjoint $a^*$ is essentially coercive.
However, a surprising consequence of Theorem~\ref{th:4.1} is that, for an essentially coercive form, uniqueness 
for the form and uniqueness for its  adjoint are equivalent, as the following corollary shows. 
\begin{cor}
	Let $\V$ be a separable Hilbert space on $\K$ and $a:\V\times\V\to\K$ be a  continuous essentially coercive form. The following assertions are equivalent:
	\begin{enumerate}
		\item[(i)] for all $u\in\V$, $a(u,v)=0$ for all $v\in\V$ implies $u=0$;
		\item[(ii)] for all $v\in\V$, $a(u,v)=0$ for all $u\in\V$ implies $v=0$;
		\item[(iii)] for all $L\in\V'$ there exists $u$ in $\V$ such that 
		$a(u,v)=\langle L ,v\rangle$, for all $v\in\V$;
		\item[(iv)]  for all $L\in\V'$ there exists $v$ in $\V$ such that 
		$a(u,v)=\overline{\langle L ,u\rangle}$, for all $u\in\V$.
	\end{enumerate}
\end{cor}
\begin{proof}
$(i)\Longleftrightarrow (ii)$: this follows from Theorem~\ref{th:4.1} and Theorem~\ref{th:3.2}. The other equivalences follow from Corollary~\ref{cor:2.3}. 	
\end{proof}	
\section{The Aubin-Nitsche trick revisited}\label{sec:5}
  In this section we want to prove  that on suitable Hilbert spaces containing the space $\V$ continuously the approximation speed in the Galerkin approximation can be improved. 
  We refer also to \cite{schatz-wang} for related, but different results in this direction.

  Let $\V$ be a separable Hilbert space over $\K=\R$ or $\C$, and $a:\V\times\V\to\K$ a sesquilinear form satisfying 
  \[  |a(u,v)|\leq M \|u\|_\V  \|v\|_\V.  \] 
  Let $(\V_n)_{n\in\N}$ be an approximating sequence of $\V$. We assume that (BNB) holds; i.e. there exists $\beta>0$ such that
  \begin{equation}\label{eq:5.1}
  \hbox{For all }n\in\N,~
  \sup_{v\in\V_n,\|v\|_\V=1}|a(u,v)|\geq \beta \|u\|_{\V} \mbox{ for all }u\in\V_n. 
  \end{equation} 
  Given $L\in\V'$ and $n\in\N$, let $u_n\in\V_n$ be the solution of 
  \begin{equation}\label{eq:5.2}
   a(u_n,\chi)=\langle L,\chi \rangle \mbox{ for all }\chi\in\V_n,
   \end{equation}
   and $u\in\V$ the solution of 
   \begin{equation}\label{eq:5.3}
   a(u,v)=\langle L,v\rangle \mbox{ for all }v\in \V.
   \end{equation}
   Note that, by subtracting \eqref{eq:5.3} and \eqref{eq:5.2}, we obtain the following \emph{Galerkin orthogonality}:
   \begin{equation}\label{eq:5.8}
   a(u-u_n,v)=0 \mbox{ for all }v\in \V_n.
   \end{equation}
   We know from Proposition~\ref{prop:3.3}  and Proposition~\ref{prop:3.3KatoArendt} that 
   \begin{equation}\label{eq:5.4}
   \|u-u_n\|_\V\leq \frac{M}{\beta} \dist(u,\V_n)
   \end{equation} 
   for all $n\in\N$. 
   We want to improve this estimate if the given data $L\in\V'$ is in a suitable subspace of $\V'$. 
   
   Let $\X\hookrightarrow \V'$; i.e. $\X$ is a Banach space such that $\X\subset \V'$ and 
   \[  \|f\|_\X\leq c_\X \|f\|_{\V'}  \] 
   for all $f\in\X$ and some $c_\X>0$. We define for $n\in\N$   
 
   \begin{equation}\label{eq:5.5}
   \gamma_n(\X):=\sup_{f\in\X,\|f\|_\X =1} \dist ({\A}^{-1}f,\V_n),
   \end{equation}
   where the distance is taken in $\V$. 
   Thus 
   \begin{equation}\label{eq:5.6}
   \dist (w,\V_n)\leq \gamma_n(\X)\|\A w\|_\X\mbox{ for all }w\in {\A}^{-1}\X,
   \end{equation}
   where $\A:\V\to \V'$ is the isomorphism given by
   \[\langle \A u,v\rangle =a(u,v). \]
   Thus, if $u$ is the solution of (\ref{eq:5.3}) and $u_n$ the approximate solution of (\ref{eq:5.2}), then, if $L\in\X$, we have the estimate 
   \begin{equation}\label{eq:23}
   \|u-u_n\|_\V \leq \frac{M}{\beta} \gamma_n (\X)\|L\|_{\X},
   \end{equation}
   which has the advantage of being uniform for $L$ in the unit ball of $\X$. 
   \begin{rem}
   	Let $Q_n:\X\to\V, L\mapsto u_n$ be the solution operator for (\ref{eq:5.2}). Then (\ref{eq:23}) says that 
   	\[ \|\A^{-1}-Q_n\|_{{\mathcal L}(\X,\V)}\leq \frac{M}{\beta} \gamma_n(\X).   \]
   \end{rem}
We can characterize when $\gamma_n(\X)\to 0$ as $n\to\infty$. 
  % We hope that this improves estimate (\ref{eq:5.4}) if $\X$ is well chosen. This is indeed the case as the following proposition shows. 
   \begin{prop}\label{prop:5.1}
   	One has 
   	\[\lim_{n\to \infty}\gamma_n(\X)=0\mbox{ if and only if }\X\hookrightarrow \V'\mbox{ is compact}.\]
   \end{prop}
\begin{proof}
Denote by $P_n:\V\to\V_n$ the orthogonal projection onto $\V_n$. Then 
\[\gamma_n(\X)=\sup_{f\in\X,\|f\|_\X =1}\| {\A}^{-1}f-P_n{\A}^{-1}f\|_{\V}=\|{\A}^{-1}\circ j-P_n  {\A}^{-1}\circ j\|_{{\mathcal L}(\X,\V)},  \]
where $j:\X\to \V'$ is the canonical injection. If $j$ is compact, then $K:={\A}^{-1}\circ j(B_\X)$, where $B_\X$ is the unit ball of $\X$, is relatively compact in $\V$. Now, $P_n$ converges strongly to the identity of $\V$. Since $\|P_n\|\leq 1$, this convergence is uniform on compact subsets of $\X$. This shows that $\gamma_n(\X)\to 0$ as $n\to\infty$. 

Conversely, if $\gamma_n(\X)\to 0$, then ${\A}^{-1}\circ j$ is compact as limit of finite rank operators. Then also $j$ is compact.        
\end{proof}
Similarly, we define 
\[\gamma_n^*(\X):=\sup_{f\in\X,\|f\|_\X =1}\dist ( {\A}^{*-1}f,\V_n ), \]
where $\A^*:\V\to\V^*$ is given by 
\[   \langle \A^* u,v\rangle =a^* (u,v):=\overline{a(u,v)}.   \] 
As before we have  $\gamma_n^*(\X)$  defined as $\gamma_n(\X)$ but with $a$ replaced by the adjoint form $a^*$ of $a$.  Thus we have for all 
$w\in{\A}^{*-1}\X$, 
\begin{equation}\label{eq:n2.7}
\dist(w,\V_n)\leq \gamma_n^*(\X)\|\A^*w\|.
\end{equation}
Now we apply the Aubin--Nitsche trick in the following proof. In contrast to the literature    \cite{EG04} we allow non-selfadjoint forms and also let $L\in\X$ where $\X\hookrightarrow \V'$ is arbitrary. However, as usual, we fix a Hilbert space $\GH$ such that $\V\hookrightarrow\GH$ with dense range. Thus we have the Gelfand triple

\[    \V \hookrightarrow \GH\hookrightarrow \V'.\]
Now we let $\X\hookrightarrow\V'$ be another Banach space in which we choose the given data $L$, whereas our error estimate is done with respect to the norm of $\GH$. 

\begin{thm}\label{th:n6.2}
Let $L\in\X$ and let $u$ be the solution of (\ref{eq:5.3}), $u_n$ the solution of (\ref{eq:5.2}). Then 
\begin{equation}\label{eq:n24}
\|u-u_n\|_{\GH}\leq \frac{M^2}{\beta}\gamma_n(\X)\gamma_n^*(\GH)\|L\|_{\X},
\end{equation} 
for all $n\in\N$. 
\end{thm}
\begin{proof}
Let $n\in\N$. Then, on the footsteps of  Aubin--Nitsche, we consider the solution $w\in\V$ of 
\begin{equation}\label{eq:n25}
a^*(w,v)=\langle u-u_n,v\rangle_{\GH}\hspace{0,5cm} (v\in\V).
\end{equation}  
Then, by (\ref{eq:n25}), for any $\chi\in\V_n$, 
\begin{eqnarray*}
	\|u-u_n\|_{\GH}^2 & = & \langle u-u_n,u-u_n\rangle_{\GH}= a^*(w,u-u_n)=\overline{a(u-u_n,w)}\\
	                            & = & \overline{a(u-u_n,w-\chi)}\leq M\|u-u_n\|_\V \|w-\chi\|_\V   
\end{eqnarray*}	
where in the last identity we used the Galerkin orthogonality (\ref{eq:5.8}).  	

Since $\chi\in\V_n$ is arbitrary, this implies that 
\[  \|u-u_n\|_{\GH}^2  \leq M\|u-u_n\|_\V \dist(w,\V_n).\]
Now we use (\ref{eq:23}) and (\ref{eq:n2.7})  	 to deduce 
\[  \|u-u_n\|_{\GH}^2\leq M.\frac{M}{\beta}\gamma_n(\X)\|L\|_\X\gamma_n^*(\GH)\|u-u_n\|_{\GH}.   \]
Consequently, we obtain
\[\|u-u_n\|_{\GH}\leq \frac{M^2}{\beta} \gamma_n(\X)\gamma_n^*(\GH) \| L \|_{\X}.\]	
	
\end{proof}	
	
\section{Applications}

\subsection{Selfadjoint positive operators with compact resolvent}

As an illustration, we apply Theorem~\ref{th:n6.2} 
 %We demonstrate the results by applying them
  to selfadjoint positive operators with compact resolvent. Let $\V,\GH$ be infinite dimensional, separable Hilbert spaces over $\K=\R$ or $\C$ such that $\V$ is compactly injected in $\GH$ and dense in $\GH$. Thus we have the Gelfand triple
 \[ \V\hookrightarrow \GH\hookrightarrow \V' . \]       
 Let $a:\V\times\V\to \K$ be continuous, symmetric and coercive. Then the operator $\A:\V\to\V'$ given by 
 \[  \langle \A u,v\rangle =a(u,v)\] is invertible. Moreover, there exist an orthonormal basis $(e_n)_{n\geq 0}$ of $\GH$ and $\lambda_n\in\R$ such that 
 \[ 0<\lambda_0\leq \lambda_1\leq \cdots ,\lim_{n\to\infty}\lambda_n=\infty   \]
 and 
 \[  \V=\{ u\in\GH :\sum_{n=0}^\infty \lambda_n |\langle u,e_n\rangle_\GH |^2<\infty \}   \]
 (see e.g. \cite[Satz 4.49]{AU19})   and 
 \[a(u,v)=   \sum_{n=0}^\infty  \lambda_n \langle u,e_n\rangle_\GH \langle e_n,v\rangle_\GH .   \]
 Passing to an equivalent scalar product we may and will assume that 
 \[  \langle u,v\rangle_\V    =a(u,v)\;\;\;\; (u,v\in\ V).  \] 
 Thus $|a(u,v)|\leq \|u\|_{\V}  \|v\|_{\V}$ and $\sup_{  \|v\|_\V= 1}  |a(u,v)|=\|u\|_\V$; i.e. we have $M=\beta=1$ in the above estimates.
 
 Consider $\V_n=\span \{ e_0,\cdots,e_{n-1}\},n=1,2,\cdots$. Then $(\V_n)_{n\in\N}$ is an approximating sequence of $\V$. We define for $s\in[-1,1]$ 
 \[  \V_s:=\{   f\in\V':     \sum_{n\geq 0}  \lambda_n^s |\langle f,e_n\rangle_\GH |^2<\infty        \},  \]
 which is a Hilbert space for the norm 
 \[ \|f\|^2_{\V_s} =\sum_{n\geq 0}\lambda_n^s  |\langle f,e_n\rangle_\GH|^2.    \]
 Then it is easy to see that $\V_{-1}=\V'$, $\V_0=\GH$, $\V_1=\V$ with identity of the norms. Morever, for $s\in (0,1)$, 
 \[    \V_s =(\V_0,\V_1)_s \]
 (the complex interpolation space) and for $s\in (-1,0)$, 
 \[  \V_s =(\V_{0},\V_{-1})_{-s}. \]  
 \begin{lem}
 	One has for $s\in [-1,1]$, 
 	\[   \gamma_n (\V_s)=|\lambda_n|^{-(1+s)/2}\;\;\;\; (n=1,2,\cdots).  \]
 	In particular,
 	\[\gamma_n (\GH)=|\lambda_n|^{-1/2}.   \] 
 \end{lem}
\begin{proof}
Let $\widehat{e_n}=\frac{1}{\sqrt{\lambda_n}}e_n$. Then  $(\widehat{e_n})_{n\geq 0}$ is an orthonormal basis of $\V$. For $u\in\V$,
\[   \langle u,\widehat{e_k}\rangle_{\V}= \sum_{n\geq 0}  \lambda_n \langle u,e_n\rangle_{\GH} \langle e_n,\widehat{e_k}\rangle_{\GH}=\sqrt{\lambda_k} \langle u,{e_k}\rangle_{\GH} \]  
Thus 
\[   P_n u=\sum_{k=0}^{n-1}  \langle u,\widehat{e_k}\rangle_{\V} \; \widehat{e_k}=  \sum_{k=0}^{n-1}  \langle u,{e_k}\rangle_{\GH} \;{e_k}\]
defines the  orthogonal projection of $\V$ onto $\V_n$. Moreover, in $\V$ one has 
\[  \dist(u,\V_n)^2 =\|u-P_nu\|_{\V}^2=\sum_{k\geq n}\lambda_k | \langle u,e_k \rangle_{\GH}|^2.  \]    
Let $f\in \V_s$, $u={\A}^{-1}f$. Then 
\[  \langle f,e_k\rangle_{\GH}=\langle \A u,e_k\rangle_{\GH}=\lambda_k\langle u,e_k\rangle_{\GH}.  \]
Thus 
\begin{eqnarray*}
{\gamma_n({\V}_s)}^2 & = &\sup_{\A u=f}\frac{\|u-P_n u\|_{\V}^2}{\|f\|_{\V_s}^2}	=
\sup_{f\in\V_s, \A u=f} \frac{ \sum_{k\geq n}\lambda_k |\langle u,e_k\rangle_{\GH}|^2}{\sum_{k\geq 0}\lambda_k^s |\langle f,e_k\rangle_{\GH}|^2}    \\
 & = & \sup_{f\in\V_s} \frac{ \sum_{k\geq n}\lambda_k^{-1} |\langle f,e_k\rangle_{\GH}|^2}{\sum_{k\geq 0}\lambda_k^s |\langle f,e_k\rangle_{\GH}|^2} =
  \sup_{f\in\V_s} \frac{ \sum_{k\geq n}\lambda_k^{-1-s}\lambda_k^s |\langle f,e_k\rangle_{\GH}|^2}{\sum_{k\geq 0}\lambda_k^s |\langle f,e_k\rangle_{\GH}|^2}  \\
   & \leq & \lambda_n^{-1-s}
\end{eqnarray*}
since $(\lambda_k)_{k\geq 0}$ is increasing. 

Taking $f=e_n$, one sees that $\gamma_n(\V_s)^2\geq \frac{\lambda_n^{-1}}{\lambda_n^s}=\lambda_n^{-s-1}$. 
\end{proof}

Now let $f\in\V_s$, where $-1\leq s\leq 1$ and let $u=\A^{-1}f$. Let $u_n\in\V$ such that 
\[  a(u_n,\chi)=\langle f,\chi\rangle  \;\;\;\; (\chi\in\V_n),  \]
i.e. $u_n$ is the approximate solution. Then by Theorem~\ref{th:n6.2}
\[   \|u-u_n\|_{\GH}\leq \gamma_n(\X)\gamma_n(\GH)  \|f\|_{\V_s}=|\lambda_n|^{-(1+s)/2}|\lambda_n|^{-1/2}\|f\|_{\V_s}. \]  
Thus we obtain the following error estimate
\vspace{0,2cm}
\begin{equation}\label{eq:n26b}
\|u-u_n\|_{\GH}\leq |\lambda|^{-1-s/2}\|f\|_{\V_s}.
\end{equation}
\vspace{0,2cm}

\begin{rem}
	In this special case one can compute the error directly. In fact $u=\sum_{k=0}^\infty \frac{1}{\lambda_k}\langle f,e_k\rangle_{\GH}e_k$ and $u_n=\sum_{k=0}^{n-1}\langle f,e_k\rangle_{\GH}e_k$. Thus 
	\[  \|u-u_n\|_{\GH}^2=\sum_{k=n}^\infty \frac{1}{\lambda_k^2}  |\langle f,e_k\rangle_{\GH}|^2=  \sum_{k=n}^\infty \lambda_k^{-2-s}\lambda_k^s  |\langle f,e_k\rangle_{\GH}|^2\leq \lambda_n^{-2-s}\|f\|^2_{\V_s},  \]
	which is exactly the estimate (\ref{eq:n26b}). This means that Theorem~\ref{th:n6.2} gives the best possible estimate of the error.   
\end{rem}

Let us provide an example of application of (\ref{eq:n26b}). Let $\K=\C$, $\GH=L^2(0,2\pi)$ with norm $\|u\|_{\GH}^2=\frac{1}{2\pi}\int_{0}^{2\pi}|u(t)|^2dt$. Let $\V=\{ u\in H^1(0,2\pi) :u(0)=u(2\pi)\}$ with norm 
\[  \|u\|_{\V}^2=\frac{1}{2\pi}\int_0^{2\pi}|u'(t)|^2dt+ \frac{1}{2\pi} \int_0^{2\pi} |u(t)|^2dt. \]	
Then the injection $\V\hookrightarrow \GH$ is compact. Let $a:\V\times\V\to\C$ be given by 
\[  a(u,v)=\frac{1}{2\pi} \int_0^{2\pi} u'(t)\overline{v'(t)}dt + \frac{1}{2\pi}\int_0^{2\pi} u(t)\overline{v(t)}dt. \] 
Let $f\in L^2(0,2\pi)$. Then there exists a unique $u\in H^2(0,2\pi)$ such that 
\[    u-u''=f,\;\;\; u(0)=u(2\pi), \;\;\; u'(0)=u'(2\pi).   \]
In fact, $u$ is the unique element of $\V$ such that $a(u,v)=\langle f,v\rangle$ for all $v\in\V$. 

For $u\in\GH$, let $\widehat{u}(k)=\frac{1}{2\pi} \int_0^{2\pi}u(t)e^{-ikt}dt$ be the $k$-th Fourier coefficient. Then 
\[   a(u,v)=\sum_{k\in\Z}  (1+k^2)\widehat{u}(k)\overline{\widehat{v}(k)}. \] 
Let $e_k(t)=e^{ikt},t\in(0,2\pi)$. Then $(e_k)_{k\in\Z}$ is an orthonormal basis of $\GH$ and $\widehat{u}(k)=\langle u,e_k\rangle_{\GH}$. Let $\V_n=\span \{e_k:|k|< n\}$ and let $u_n$ be the approximate solution i.e. 
\[ a(u_n,\chi) =\langle f,\chi\rangle_{\GH} \;\;\; (\chi\in\V_n).\]
Then our estimate shows that 
 \[  \|u_n-u\|_{L^2}\leq  \frac{1}{(1+n^2)^{1/2}}\|f\|_{L^2}.  \]
 Let $0<s\leq 1$ and $\V_s:=\{  u\in L^2(0,2\pi) :\sum_{k\in\Z}(1+k^2)^s    |\widehat{u}(k)|^2<\infty\}$. If $f\in \V_s$, then by (\ref{eq:n26b}), 
 \begin{equation}\label{eq:n27}
\|u-u_n\|_{L^2}\leq (1+n^2)^{-1-s/2}\|f\|_{\V_s}. 
 \end{equation}

 \subsection{Finite elements for the Poisson problem }\label{sec:6}    

%\subsection{ Finite elements for the Poisson problem}  
  In this section we want to apply our results to show the convergence of a numerical approximation 
  via triangularization for the solution of a Poisson problem where coercivity is violated but essential coercivity   holds. For simplicity we choose $\K=\R$ throughout this section.    Let $\Omega\subset \R^d $ be an open, bounded, convex set and let $a_{ij}:\Omega\to\R$ ($1\leq i,j\leq d$) be Lipschitz continuous functions such that 
  \[a_{ij}=a_{ji}\mbox{ and }\sum_{i,j=1}^d a_{ij}(x)\xi_i\xi_j\geq \alpha |\xi|^2\quad (\xi\in\R^d)\]
  for all $x\in \Omega$, where $\alpha>0$. 
  Moreover, let $b_j,c_j\in W^{1,\infty}(\Omega)$ for $j=1,\cdots, d$ and $b_0\in L^\infty(\Omega)$. We consider the operator $A$ given by 
  \[ Au:=-\sum_{i,j=1}^d D_i(a_{ij}D_ju) +\sum_{j=1}^d b_jD_ju - \sum_{j=1}^d D_j(c_ju) + b_0u \quad (u\in H^2(\Omega)).  \]
Note that $A:H^2(\Omega)\to L^2(\Omega)$ is linear and continuous. 

Our aim is to study the Poisson equation
\begin{equation}\label{eq:5}
Au=f
\end{equation}
where $f\in L^2(\Omega)$ is given and a solution $u\in H_0^1(\Omega) \cap H^2(\Omega)$ is to be determined and calculated by approximation. We will impose the uniqueness condition
\begin{equation}\label{eq:6}
\hbox{For all }u\in H_0^1(\Omega)\cap H^2(\Omega),~ Au=0\mbox{ implies }u=0.
\end{equation}
We use the continuous, coercive form
\[   a_0:H^1_0(\Omega)\times H^1_0(\Omega)\to \R   \]
given by 
\[  a_0(u,v)=\sum_{i,j=1}^d\int_{\Omega} a_{ij}D_j uD_iv    \]
and also the perturbed form $a$ given by 
\[  a(u,v)=a_0(u,v) +\sum_{j=1}^d\int_{\Omega} (b_jD_ju v+  c_juD_jv) +  \int_\Omega b_0 uv.   \] 
Note that the adjoint form $a^*$ defined by $a^*(u,v)=a(v,u)$ has the same form as $a$. This is the reason why we also consider the coefficients $c_j$. 

 Then the following well posedness result holds.
 \begin{thm}\label{th:2.1}
 	\quad \\
 	\vspace{-0,5cm}
 	\begin{enumerate}
 	\item[i)] The form $a$  is essentially coercive. 	
	\item[ii)] Assume (\ref{eq:6}). Then for each $f\in L^2(\Omega)$ there exists a unique solution $u\in H^1_0(\Omega)\cap H^2(\Omega)$ of (\ref{eq:5}). 
	\end{enumerate}
\end{thm} 

\begin{proof}
a) We first show $H^2$-regularity. Let $u\in H^1_0(\Omega)$, $f\in L^2(\Omega)$ such that $a(u,v)=\int_{\Omega}fv$ for all $v\in H^1_0(\Omega)$. Then $u\in H^2(\Omega)$ and $Au=f$. In fact, let 
\[  g:=f-b_0 u -\sum_{j=1}^d(b_jD_j u -D_j(c_j u)).\]
 Then $g\in L^2(\Omega)$ and $a_0(u,v)=\int_{\Omega}gv$ for all $v\in H^1_0(\Omega)$.   Now it follows from the classical $H^2$-result of Kadlec \cite{kadec} (see \cite[Theorem 3.2.1.2]{G85}) that $u\in H^2(\Omega)$. It clearly follows that $Au=f$.  	\\
 b) We show that $a$ is essentially coercive. Let $u_n\rightharpoonup 0$ as $n\to\infty$ in  $H_0^1(\Omega)$ and $a(u_n,u_n)\to 0$ as $n\to\infty$. Then $D_ju_n\rightharpoonup 0$ as $n\to\infty$ in $L^2(\Omega)$. Since the embedding of $H_0^1(\Omega)$ in $L^2(\Omega)$ is compact, it follows that $u_n\to 0$ in $L^2(\Omega)$. Consequently 
 \[    \int_{\Omega}  b_jD_ju_n.u_n\to 0,\int_\Omega c_j D_j u_n.u_n\to 0\mbox { and }\int_{\Omega} b_0 u_n.u_n\to 0\mbox{ as }n\to\infty. \] 
 Thus also $a_0(u_n,u_n)\to 0$ as $n\to\infty$. Since $a_0$ is coercive this implies $\|u_n\|_{H^1} \to 0$ as $n\to\infty$. \\
 c) The form $a$ satisfies uniqueness. In fact, let $u\in H_0^1(\Omega)$ such that $a(u,v)=0$ for all $v\in H_0^1(\Omega)$. Then $u\in H^2(\Omega)$ by part a) of the proof. Hence $u=0$ by our assumption  (\ref{eq:6}).\\
 d) Let $f\in L^2(\Omega)$. It follows from Corollary~\ref{cor:2.3} that there exists a unique $u\in H_0^1(\Omega)$ such that $a(u,v)=\langle f,v\rangle_{L^2}$ for all $v\in H_0^1(\Omega)$. Now a) implies that $u\in H^2(\Omega)$ and $Au=f$.   
 
\end{proof}	
Concerning the uniqueness property, we make the following remark.    
\begin{rem}[Eigenvalues and uniqueness]
Replace the operator $A$ by $A_\lambda:=A-\lambda\Id$ (i.e. $b_0$ by $b_0-\lambda$) where $\lambda\in\R$. Then there exists  a finite or countable infinite set such that 
\[   \{\lambda:(\ref{eq:6})  \mbox{ is violated for }A_\lambda \}=\{ \lambda_n:n\in\N,n<N\}\] 
where $1<N\leq \infty$ and $\lambda_n\in\R$, $\lim_{n\to \infty}\lambda_n=\infty$ if $N=\infty$.\\
If $b_1=\cdots=b_d=c_1=\cdots=c_d=0$ and $b_0\geq 0$, then $\lambda_n>0$ for all $n\in\N$ and then we are in the coercive case. But in general there will be also negative eigenvalues. 	The uniqueness condition (\ref{eq:6}) for $A$ is equivalent to saying that $\lambda_n\neq 0$ for all $n\in\N$.   
\end{rem}  
 Our final aim is to show that the finite element method yields an approximation of the solution of (\ref{eq:5}). 
 
 For that purpose we assume that $d=2$ and that $\Omega$  is a convex polygon. 
 Let $\{\tau_h\}_{h>0}$ be  a quasi-uniform admissible triangularization of $\Omega$ (see \cite[Definition 9.26]{AU19}). In particular each $\tau_h$ consists of finitely many triangles covering $\Omega$ of outer radius $r_T\leq h$. 
 
 For $h>0$,  we consider the corresponding finite element space $V_h$ (see \cite[Equation (9.35)]{AU19}). Thus $V_h$ consists of those continuous functions on $\overline{\Omega}$ which vanish   at $\partial\Omega$ and are affine on each triangle $T\in\tau_h$. 
 
 The following fundamental estimates are classical (see e.g. \cite[Korollar 9.28]{AU19}) . 
 \begin{prop}\label{prop:2.3}
 There exists a constant $c>0$ such that for all $h\in (0,1)$ and for each $v\in H^2(\Omega)$,
 \begin{equation}\label{eq:8}
 \inf_{\chi\in V_h}\| v-\chi  \|_{H^1 (\Omega)}\leq c h|v|_{H^2(\Omega)},
\end{equation}
where $|v|^2_{H^2(\Omega)}:=\int_{\Omega}(|D_1^2v|^2+2|D_1D_2v|^2 +|D_2v|^2)$. 
 \end{prop}
Note that Proposition~\ref{prop:2.3} shows how we can approximate functions in $H^2(\Omega)$ by finite elements and so far there is no relation with the solutions of the Poisson equation. 

We assume the uniqueness condition (\ref{eq:6}). Then by Theorem~\ref{th:4.1}, since the form $a$ is essentially coercive,   there exists $h_0\in (0,1]$ such that for $0<h\leq h_0$ and $u\in V_h$
\begin{equation}\label{eq:9}
a(u,\chi)=0\mbox{ for all }\chi\in V_h\mbox{ implies }u\in V_h.
\end{equation}  
Let $f\in L^2(\Omega)$. Since $V_h$ is finite dimensional, it follows from (\ref{eq:9}) that for all $0<h\leq h_0$, there exists a unique $u_h\in V_h$ such that
\begin{equation}\label{eq:10}
a(u_h,\chi)=\int_{\Omega}f\chi \mbox{ for all }\chi\in V_h.
\end{equation}    
The finite elements $(u_h)_{0<h\leq h_0}$ are the approximation of the solution of (\ref{eq:5}) we are interested in. They converge in $H^1(\Omega)$ with convergence order $1$ and in $L^2(\Omega)$ with convergence order $2$. More precisely, the following is our main theorem of this section. 
 \begin{thm}\label{th:2.4}
 Let $f\in L^2(\Omega)$ and consider the approximate solutions $u_h$, $0<h\leq h_0$. Then there exist $0<h_1\leq h_0$ and constants $c_1,c_2$ independent of $f$ such that
 \begin{equation}\label{eq:7.7}
 \|u-u_h\|_{H^1(\Omega)}\leq c_1h \|f\|_{L^2(\Omega)}
 \end{equation} 
 and 
 \begin{equation}\label{eq:7.8}
 \|u-u_h\|_{L^2(\Omega)}\leq c_2 h^2 \|f\|_{L^2(\Omega)}
 \end{equation}
 where $u$ is the solution of (\ref{eq:5}). 
 \end{thm}
\begin{proof}
Applying the closed graph theorem in the situation of Theorem~\ref{th:2.1}, we find a constant $c_3>0$ such that 
\begin{equation}\label{eq:13}
\|u\|_{H^2(\Omega)}\leq c_3 \|f\|_{L^2(\Omega)}
\end{equation} 	
whenever $f\in L^2(\Omega)$ and $u$ solves (\ref{eq:5}). 

By Theorem~\ref{th:4.1}, there exist $\gamma>0$, $0<h_1\leq h_0$, both independent of $f$, such that 
\[  \|u-u_h\|_{H^1(\Omega)}\leq \gamma \inf_{\chi\in V_h} \|\chi -u\|_{H^1(\Omega)}  \]
for all $0<h\leq h_1$. Thus (\ref{eq:8}) implies that for $0<h\leq h_1$, 
\[\|u_h-u\|_{H^1(\Omega)}\leq ch\gamma|u|_{H^2(\Omega)}.\]  
Now (\ref{eq:7.7}) follows from (\ref{eq:13}). 

Next we establish the $L^2$-estimate (\ref{eq:7.8}). For that we compute using (\ref{eq:8}),
\[   \gamma_h(\GH)=\sup_{w\in H_0^1(\Omega)\cap H^2(\Omega)} \frac{\dist(w,\V_h)}{\|  A w\|_{L^2(\Omega)}}  \leq \sup_{w\in H_0^1(\Omega)\cap H^2(\Omega)}\frac{ch |w|_{H^2(\Omega)}}{\| A w\|_{L^2(\Omega)}} .  \]
%It follows from the closed graph theorem that there exists a constant $c_3>0$ such that $\|w\|_{H^2(\Omega)} \leq c_3\|f\|_2$ whenever $w\in %H^2(\Omega)\cap H_0^1(\Omega), Aw=f$. 
Since $|w|_{H^2(\Omega)}\leq \|w\|_{H^2(\Omega)}$, it follows from (\ref{eq:13}) that $\gamma_h(\GH)\leq cc_3 h$ for all $h>0$. 

The same estimate is true for $\gamma_h^*(\GH)$. Now assume that (\ref{eq:7.8}) is false. Then there exists a sequence $h_n\downarrow0$ as $n\to\infty$ such that (\ref{eq:7.8}) does not hold for all $h=h_n$ and any constant $c_2$. This contradicts Theorem~\ref{th:n6.2}. 
\end{proof}	

\begin{rem}
There are other methods to approximate the solution of a non-coercive advection-diffusion equation as (\ref{eq:5}). In fact, Le Bris, Legoll and Madiot \cite{BLM16} use the Banach-Ne\v cas-Babuska lemma (instead of essential coercivity as we do) and a special measure to construct an approximation. 

The advantage is that no initial mesh $h_1$ has to be considered; on the other hand there seems to be no such precise error estimate as our quadratic convergence obtained in Theorem~\ref{th:2.4} even though numerical examples are given in \cite{BLM16}. 

Still, another approach (based on Fredholm perturbation) is presented by Christensen \cite{ch03}, which also involves the Babuska inf-sup condition. 

Finally, let us mention the works by Droniou, Gallou\"et and Herbin \cite{dgh03}, based on finite volume methods, which also present the advantage to provide an approximate solution for this problem on any admissible mesh. 

One of the first results on the Galerkin method in a special non-coercive case are due to Schatz \cite{schatz} and Schatz--Wang \cite{schatz-wang}.    
\end{rem}

\section{Supplement: saddle point problems}\label{sec:8}

Brezzi's contribution \cite{brezzi} is a version of (BNB) which implies the convergence of the Galerkin approximation in the case of saddle point problems. Let us consider the case where  $\W$ and $\Y$ are real Hilbert spaces and $\widehat{a}~:~\W\times\W \to {\mathbb R}$ and  $\widehat{b}~:~\W\times\Y \to {\mathbb R}$ are continuous bilinear forms  in the sense that there exists $M>0$ with
\[   | \widehat{a}(w,v)|\leq M \|w\|_{\W}\| v \|_{\W}\mbox{ for all }w\in \W,v\in \W  \] 
and
\[   | \widehat{b}(w,q)|\leq M \|w\|_{\W}\| q \|_{\Y}\mbox{ for all }w\in \W,q\in \Y . \] 
Then, given $(f,g)\in \W'\times\Y'$,  the \emph{continuous saddle point problem} consists in finding $(w,p)\in \W\times\Y$ such that
\begin{eqnarray*}
& \forall z\in\W,&  \widehat{a}(w,z)+\widehat{b}(z,p) = f(z),\\
& \forall q\in\Y,&  \widehat{b}(w,q) = g(q).
\end{eqnarray*}

\begin{exam}
An important example is the Stokes problem (motivating some investigation  by Lady\v zhen\-skaya \cite{ladyzhenskaya}), with $\W = (H^1_0(\Omega))^d$, where $d$ is the space dimension, $\Y = L^2_0(\Omega)$ (the space of $L^2$-functions with null average), 
\[\widehat{a}(w,z) = \sum_{i=1}^d\int_{\Omega}\nabla w^{(i)}(x)\cdot\nabla z^{(i)}(x){\rm d}x\]
and \[
     \widehat{b}(z,p) = \int_{\Omega}(\nabla\cdot z)(x) p(x){\rm d}x.\]
\end{exam} 

\medskip

The approximation of the saddle point problem is then generally done by a mixed method \cite{brezzi,brezzifortin}, letting, for $n=1,2,\ldots$, $(\W)_{n\in{\mathbb N}^\star}$ and $(\Y)_{n\in{\mathbb N}^\star}$ be approximating sequences in the spaces $\W$ and $\Y$, respectively,  in the sense of Definition~\ref{def:appseq}, and looking for $(w_n,p_n)\in \W_n\times\Y_n$ such that
\begin{eqnarray*}
& \forall z\in\W_n,&  \widehat{a}(w_n,z)+\widehat{b}(z,p_n) = f(z),\\
& \forall q\in\Y_n,&  \widehat{b}(w_n,q) = g(q).
\end{eqnarray*}
We call this the \emph{approximate saddle point problem.}\\

The following result shows that conditions \eqref{eq:infsupwzp} (which are Brezzi's conditions \cite[Hypotheses H1 and H2]{brezzi}) are sufficient for the convergence of the solutions of the approximate saddle point problems. This is proved by Brezzi  \cite[Theorem 2.1]{brezzi}, where a solution is assumed to exist. However, similar to the proof of our Proposition \ref{prop:3.3}, one can show that Brezzi's conditions imply existence and uniqueness of the continuous saddle point problem. Indeed,  following the proof of \eqref{eq:brezzithm} given in the proof of  \cite[Theorem 2.1]{brezzi}, letting $w=0$ and $p=0$, we get a bound on the approximate solution, and a solution of the  continuous problem can be obtained by passing to the limit of a weakly converging subsequence. Uniqueness follows from the estimate (8.2) proved by Brezzi. 
For $n=1,2,\ldots$, define 
\[
\W_{0,n} = \{ u\in \W_n;~\forall q\in \Y_n,~ \widehat{b}(u,q) = 0\}
\]
and assume that $\W_{0,n}^* :=\W_{0,n}\setminus \{0\}\neq \emptyset$ and $\Y_{n}^* :=\Y_{n}\setminus \{0\}\neq \emptyset$ for all $n \in \N$.

\begin{thm}[Brezzi]
Assume that there exists $\beta>0$ such that
\begin{equation}\label{eq:infsupwzp}\left\{\begin{array}{ll}(i) &
 \forall n\in{\mathbb N}^\star,\ \inf_{w\in \W_{0,n}^\star} \sup_{z\in \W_{0,n}^\star} \frac {\widehat{a}(w,z)} {\Vert w\Vert_{\W}\Vert z\Vert_{\W}} \ge \beta\\ (ii) &
 \forall n\in{\mathbb N}^\star,\ \inf_{z\in \W_{0,n}^\star} \sup_{w\in \W_{0,n}^\star} \frac {\widehat{a}(w,z)} {\Vert w\Vert_{\W}\Vert z\Vert_{\W}}\ge \beta\\ (iii) &
 \forall n\in{\mathbb N}^\star,\ \inf_{p\in \Y_n^\star} \sup_{z\in \W_{n}^\star} \frac {\widehat{b}(z,p)} {\Vert z\Vert_{\W}\Vert p\Vert_{\Y}} \ge \beta .
 \end{array}\right.
\end{equation}

Then, given  $(f,g)\in \W'\times\Y'$, there exists a unique solution $(w,p)$ of the continuous saddle point problem and for each $n \in \N$ a unique solution $(w_n,p_n)$ of the approximate saddle point problem. Moreover, 
\begin{equation}\label{eq:brezzithm}
\forall n\in{\mathbb N}^\star,\  \|w_n - w\|_W + \|p_n - p\|_Y \leq c \big( \dist(w, W_n) + \dist(p,Y_n)\big)
\end{equation}
where the constant $c$ depends only on $\beta$ and $M$.
\end{thm}
	The saddle point problem can be cast in our framework by letting
$\V = \U = \W\times\Y$, $u = (w,p)$, $v=(z,q)$ and  \[a(u,v) = \widehat{a}(w,z)+\widehat{b}(z,p)+\widehat{b}(w,q).\]

Given $(f,g) \in \V' = \W' \times \Y'$, define $L \in \V'$ by  \[\langle L,(z,q) \rangle = \langle f,z \rangle + \langle g,q\rangle.\] 
Then  $u = (w,p)$ is a solution of the continuous saddle point problem if and only if (\ref{eq:int1}) is satisfied.
Moreover, letting $\V_n = \U_n = \W_n\times \Y_n$, a vector $u_n = (w_n,p_n)\in \V_n$ satisfies  \eqref{eq:int2} if and only if  $(w_n, p_n)$  is a solution of the approximate saddle point problem. Thus our Theorem  \ref{th:3.2} shows that the convergence property expressed in Brezzi's Theorem is equivalent to (BNB)  for the form $a$ and the approximating sequence $(\V_n)$. We can use this to show the following converse result of Brezzi's Theorem.

\begin{thm}
Assume that, given $(f,g)\in \W' \times \Y'$,  for each $n\in  \N $, there is a unique solution $(w_n,p_n)$ of the discrete saddle point problem and
that $\sup _{n\in\N} (\|w_n\|_W + \|p_n\|_Y) < \infty.$
Then Brezzi's conditions (\ref{eq:infsupwzp}) hold.
\end{thm}

\begin{proof} We know from Theorem \ref{th:3.2} and Proposition \ref{prop:3.3} that (BNB) is satisfied  for some $\beta>0$. We endow the space $\V$ with the norm $\Vert u\Vert_{\V} = \big(\Vert w\Vert_{\W}^2+\Vert p\Vert_{\Y}^2\big)^{1/2}$ for $u = (w,p)$ (it is then a Hilbert space as well). Let $n\in{\mathbb N}^\star$ be given. We then have,
\begin{equation}\label{eq:bnbmixte}
 \forall (w,p)\in \W_n\times\Y_n,\ \sup_{(z,q)\in  \W_n\times\Y_n\setminus\{(0,0)\}} \frac {|a((w,p),(z,q))|} {\Vert(z,q)\Vert_{\V}}\ge \beta \Vert(w,p)\Vert_{\V}.
\end{equation}
Let us first choose, for any $p\in \Y_n^\star$, $u=(0,p)$, which means that $w=0\in\W_n$. Let $(z,q)\in \W_n\times\Y_n\setminus\{(0,0)\}$ attaining the supremum value in \eqref{eq:bnbmixte}. We then have, from the definition of $a$ in this framework of a saddle point problem,
\[
 \frac {|\widehat{b}(z,p)|} {\big(\Vert z\Vert_{\W}^2+\Vert q\Vert_{\Y}^2\big)^{1/2}} \ge \beta \Vert p\Vert_{\Y},
\]
which implies that $z\neq 0$ and
\[
 \frac {|\widehat{b}(z,p)|} {\Vert z\Vert_{\W}} \ge \beta \Vert p\Vert_{\Y}.
\]
This proves \eqref{eq:infsupwzp}.$(iii)$ , and thus that the operator $\widehat{\mathcal B}_n~:~\W_n\to \Y_n$, defined for all $z\in \W_n$ by
\[
 \forall q\in \Y_n,\ \widehat{b}(z,q) =\langle \widehat{\mathcal B}_n z,q\rangle_{\Y},
\]
is bijective from $\W_{0,n}^\perp$ to $\Y_n$.

\medskip

Let $w\in \W_{0,n}^\star$ and let $p\in \Y_n$ be defined by
\[
 \forall q\in \Y_n,\ \langle q,p\rangle_{\Y} = \widehat{b}(\widehat{\mathcal B}_n^{(-1)}q,p) = -\widehat{a}(w,\widehat{\mathcal B}_n^{(-1)}q).
\]
Choose an element $(z,q)\in \W_n\times\Y_n\setminus\{(0,0)\}$ attaining the supremum value in \eqref{eq:bnbmixte} for this choice of $u = (w,p)$. We then write $z = z_0 + z_1$, with $z_0\in \W_{0,n}$ and $z_1\in \W_{0,n}^\perp$, which can be written as $z_1 = \widehat{\mathcal B}_n^{(-1)}q_1$ for some $q_1\in \Y_n$. We have
\[
 a((w,p),(z,q)) = \widehat{a}(w,z_0) +  \widehat{a}(w,z_1) + \widehat{b}(z_0,p)+ \widehat{b}(z_1,p) + \widehat{b}(w,q).
\]
Moreover,  $\widehat{b}(z_0,p) = \widehat{b}(w,q)=0$ since $z_0\in \W_{0,n}$ and $w\in \W_{0,n}$, and 
\[
 \widehat{a}(w,z_1) + \widehat{b}(z_1,p) = 0,
\]
by definition of $p$ and of $z_1$. Hence
\[
 \frac {|\widehat{a}(w,z_0)|} {\big(\Vert z\Vert_{\W}^2+\Vert q\Vert_{\Y}^2\big)^{1/2}} \ge \beta \big(\Vert w\Vert_{\W}^2+\Vert p\Vert_{\Y}^2\big)^{1/2}.
\]
This implies that $z_0\neq 0$, and therefore $z_0\in \W_{0,n}^\star$ is such that
\[
 \frac {|\widehat{a}(w,z_0)|} {\Vert z_0\Vert_{\W}} \ge \beta \Vert w\Vert_{\W},
\]
where we take into account that $\Vert z\Vert_{\W} \ge \Vert z_0\Vert_{\W}$ by Pythagore's theorem. This concludes the proof of \eqref{eq:infsupwzp}.$(i)$.

\medskip

The equivalence between $(BNB)$ and $(BNB^\star)$ allows to obtain the proof of \eqref{eq:infsupwzp}.$(ii)$ (with the same $\beta$, see Proposition \ref{prop:eqbetabetastar}), following the same path.

\end{proof}

In conclusion, Brezzi's conditions (\ref{eq:infsupwzp}) are equivalent to the well posedness of the continuous saddle point problem together with the convergence of the approximate solutions to the solution, and they are also equivalent to (BNB) for the form $a$ and the approximating sequence $(\V_n)$ of $\V$.

Note that \cite[Chapter II, Remark 2.11]{brezzifortin} provides a comment on the fact that \eqref{eq:infsupwzp}.$(iii)$  is a necessary condition.

\noindent \textbf{Acknowledgments:}  We are most grateful to Gilles Lancien about a discussion on the approximation property and pointing out the survey article of Casazza \cite{Cas} to us. We also thank the anonymous referee for useful and inspiring comments.
This research is partly supported by the B\'ezout Labex, funded by ANR, reference ANR-10-LABX-58.  
\bibliographystyle{abbrv}
\bibliography{horsdoeuvrefinal}
\end{document}